\documentclass[11pt]{amsart}
\usepackage[left=20mm, right=20mm, top=20mm, bottom=20mm]{geometry}
\usepackage{amssymb}
\usepackage{bm}
\usepackage{graphicx}
\usepackage[centertags]{amsmath}
\usepackage{amsfonts}
\usepackage{amsthm}
\usepackage{amsbsy}
\usepackage{mathtools}
\usepackage{mathrsfs}
\usepackage{cases}
\usepackage[all]{xy}
\usepackage{tabularx}
\usepackage{algorithm2e}
\usepackage[linktocpage=true]{hyperref}

\usepackage{amsthm}
\usepackage[most]{tcolorbox}

\tcbset{
    thmstyle/.style={
        enhanced jigsaw,     
        breakable,
        sharp corners,
        colback              = white,
        colframe             = black,
        boxrule              = 0.4pt,
        left                 = 8pt,
        right                = 8pt,
        top                  = 8pt,
        bottom               = 8pt,
        toprule at break     = 0.4pt,   
        bottomrule at break  = 0.4pt,   
        lines before break   = 3,       
        before               = \medskip\noindent,
        after                = \medskip,
        before upper = {\parindent=15pt\setlength{\belowdisplayskip}{8pt}}
    }
}

\newtheorem{theorem}{Theorem}[section]
\newtheorem{defnB}[theorem]{Definition}
\newtheorem{thmB}[theorem]{Theorem}
\newtheorem{propB}[theorem]{Proposition}
\newtheorem{lemB}[theorem]{Lemma}

\newtheorem{factB}[theorem]{Fact}
\newtheorem{claimB}[theorem]{Claim}
\newtheorem{remB}[theorem]{Remark}

\tcolorboxenvironment{theorem}{thmstyle}
\tcolorboxenvironment{defnB}{thmstyle}
\tcolorboxenvironment{thmB}{thmstyle}
\tcolorboxenvironment{propB}{thmstyle}
\tcolorboxenvironment{lemB}{thmstyle}
\tcolorboxenvironment{corB}{thmstyle}
\tcolorboxenvironment{conjB}{thmstyle}
\tcolorboxenvironment{factB}{thmstyle}
\tcolorboxenvironment{claimB}{thmstyle}
\tcolorboxenvironment{remB}{thmstyle}

\usepackage[backend=biber,style=alphabetic]{biblatex}
\hypersetup{
    colorlinks=true,
    linkcolor=blue,
    filecolor=blue,
    citecolor=blue,
    urlcolor=cyan}
\newtheorem{thm}{Theorem}[section]

\newtheorem{question}[thm]{Question}
\newtheorem{defn}[thm]{Definition}
\newtheorem{conjecture}[thm]{Conjecture}

\theoremstyle{remark}

\newcommand{\meg}{\geqslant}
\newcommand{\mik}{\leqslant}
\renewcommand{\mik}{\leqslant}
\renewcommand{\meg}{\geqslant}
\renewcommand{\leq}{\leqslant}
\renewcommand{\geq}{\geqslant}
\DeclareMathOperator*{\ave}{\mathbb{E}}
\DeclareFontFamily{U}{stix2bb}{}
\DeclareFontShape{U}{stix2bb}{m}{n} {<-> stix2-mathbb}{}

\newcommand{\prob}{\mathbb{P}}
\newcommand{\n}{\mathbb{N}}

\newcommand{\e}{\mathbb{E}}
\newcommand{\real}{\mathbb{R}}

\newcommand{\thtensor}{\boldsymbol{\theta}=\langle \theta_i: i\in [n]^d\rangle}

\newcommand{\Fcal}{\mathcal{F}}

\newcommand{\Acal}{\mathcal{A}}
\newcommand{\Mcal}{\mathcal{M}}

\newcommand{\Ccal}{\mathcal{C}}

\newcommand{\Dcal}{\mathcal{D}}

\newcommand{\Vcal}{\mathcal{V}}

\newcommand{\seminorm}[1]{{\left\vert\kern-0.25ex\left\vert\kern-0.25ex\left\vert #1
    \right\vert\kern-0.25ex\right\vert\kern-0.25ex\right\vert}}
\DeclarePairedDelimiter\abs{\lvert}{\rvert}
\DeclarePairedDelimiter\norm{\lVert}{\rVert}

\DeclarePairedDelimiterX\Set[1]\{\}{%

#1
}
\newcommand{\tailsize}[2]{\alpha_{#1,#2}}
\DeclareMathOperator{\Emb}{Emb}

\begin{document}

\title[Metric embeddings of cubes into dense subsets of cubes]{Metric embeddings of cubes into dense subsets of cubes}
\author{Miltiadis Karamanlis}
\address{Department of Mathematics, University of Athens, Panepistimiopolis 157 84, Athens, Greece}
\email[Miltiadis Karamanlis]{kararemilt@gmail.com}
\author{Cosmas Kravaris}
\address{Department of Mathematics, Princeton University, Princeton, NJ, USA}
\email[Cosmas Kravaris]{ck6221@princeton.edu}

\begin{abstract}
    \;
    Fix $k \in \mathbb{N}$ and $0 < \delta < 1$.
    We study how large $N \in \n$ must be so that for every subset of the $N$-dimensional Hamming cube $\mathcal{D} \subset Q_N$ which is $\delta$-dense, $|\Dcal|\geq \delta 2^N$, there exists some \textit{metric embedding} $f: Q_k \to \Dcal$.
    We study $3$ variants:
    For a $(1+\varepsilon)$-bi-Lipschitz map $f$ for a fixed $\varepsilon>0$,
    we show that $N = O(\varepsilon^{-2}\log(1/\delta) k^3)$.
    For an isometric map $f$ with arbitrary rescaling (i.e. undistorted), we show that $N = \log(1/\delta) e^{\Omega(k)}$. 
    For an isometric map $f$ with bounded rescaling we show that $N = \exp{[\log(1/\delta)e^{\Theta(k)}]}$.

    Regarding the path space, we prove the density analog of a coloring theorem of Rödl--Sales \cite{rodl2022blurred}.
    We give bounds for $(1+\varepsilon)$-bi-Lipschitz embeddings of the path $[k] = \{1,...,k\}$ into dense subsets of the path $[N] = \{1,...,N\}$, improving a bound of Dumitrescu \cite{dumitrescu2010approximate}.
    We prove similar bounds for the binary tree space,
    using the tree replicas theorem of Pach--Solymosi--Tardos \cite{pach2011remarks}.
    
    As a geometric application 
    we obtain a non-positive Alexandrov curvature counterpart to the work of Bartal--Linial--Mendel--Naor~\cite{BartalLinialMendelNaor2005} on the nonlinear Dvoretzky problem~\cite{BFM86}.
    By~\cite{BartalLinialMendelNaor2005}, any $\Dcal \subset Q_N$ that embeds with bi-Lipschitz distortion $<\alpha$ into a metric space of non-negative Alexandrov curvature must be small, namely, necessarily $|\Dcal| \lesssim 2^{N(1-\Omega(\alpha^{-2}))}$. 
    We prove that for every $N\gtrsim \alpha^{6}\ge 1$, any $\Dcal \subset Q_N$ that embeds with distortion $<\alpha$ into some metric space of non-positive Alexandrov curvature must satisfy $|\Dcal| \lesssim 2^{N(1-\Omega(\alpha^{-4}))}$
    via an approach which is entirely different from that of \cite{BartalLinialMendelNaor2005}. 
    We also show that nontrivial metric type and non-universality are preserved by taking finite unions of subspaces.
\end{abstract}

\maketitle
\setcounter{tocdepth}{2}


\section{\textbf{Introduction}}\label{sec:Introduction}
A central aspect of Ramsey theory is proving \textit{density theorems} which detect patterns inside every dense subset of a given combinatorial structure.
In some cases, the underlying pattern can be interpreted as a metric space and the task is to embed it into every dense subset in a manner that preserves the metric.
For example, Szemerédi's theorem \cite{szemeredi1975sets} states that for every $k \in \n$ and $0<\delta<1$ there exists $N \in \n$ such that every subset $\Dcal \subset [N]:=\{1,...,N\}$ of size $|\Dcal|>\delta N$ contains a $k$-term arithmetic progression.
A $k$-term arithmetic progression is precisely a rescaled isometric copy of the metric space $([k],|\cdot|)$ into $([N], |\cdot|)$ (where $|\cdot|$ denotes the standard metric on $\real$).
Another example is the Furstenberg--Weiss theorem \cite{furstenberg2003markov} (see also \cite{pach2011remarks}) which finds rescaled isometric copies of the binary tree metric of depth $k$ inside dense (defined in an appropriate sense) subsets of the binary tree metric of depth $N$. 
\textbf{From this metric viewpoint, it is natural to strengthen or relax the notion of embeddability and investigate how bounds sharpen or loosen, respectively}.
We refer the reader to the works of Dumitrescu \cite{dumitrescu2010approximate}, Fraser--Yu \cite{fraser2018arithmetic,fraser2021approximate}, and of Rödl--Sales \cite{rodl2022blurred} who pursued this direction.

In this paper, we study from this viewpoint the following Ramsey question: 
\textbf{Do dense subsets of the Hamming cube contain metric copies of Hamming cubes of lower dimension?}
In the bi-Lipschitz distorted setting, we also study the same question for path metrics (complementing results of \cite{dumitrescu2010approximate,fraser2018arithmetic,fraser2021approximate,rodl2022blurred}) and for tree metrics. See subsection \ref{subSect-relaxedSzemeredi} for paths and \ref{subSect-relaxedFurstenberg-Weiss} for trees.

Fix $k, N \in \n$ and let $f: (Q_k,||\cdot||_1) \to (Q_N,||\cdot||_1)$ be a map from the $k$-dimensional to the $N$-dimensional Hamming cube with their corresponding Hamming metrics $||\cdot||_1$. (See section \ref{sec:preliminaries} for basic definitions and notation.)
Denote by $\mu$ the uniform probability measure on $Q_N$.
\\-- Given $\varepsilon >0$ call $f$ a \underline{\textbf{$(1+\varepsilon)$-bi-Lipschitz map}} 
\\when there exists $r>0$ such that $r ||x-y||_1 \leq ||f(x)-f(y)||_1 \leq (1+\varepsilon) r ||x-y||_1$ for all $x,y \in Q_k$.
\\-- Call $f$ an \underline{\textbf{undistorted (or rescaled isometric) map}} \\when there exists $r>0$ such that $||f(x)-f(y)||_1 = r ||x-y||_1$ for all $x,y \in Q_k$.
\\-- Given $R \geq 1$, call $f$ a \underline{\textbf{$R$-bounded rescaling undistorted map}} 
\\when there exists $1\leq r\leq R$ such that $||f(x)-f(y)||_1 = r ||x-y||_1$ for all $x,y \in Q_k$.

Fix $k \in \n, 0 < \delta < 1, \varepsilon \geq 0, R \geq 1$.
\underline{\textbf{We denote by $\Emb(1+\varepsilon,\delta,k)$ (respectively $\Emb^{\leq R}(1,\delta,k)$)}} the smallest $N$ such that for every subset $\Dcal \subset Q_N$ with $\mu(\Dcal) \geq \delta$ (i.e. $|\Dcal| \geq \delta 2^N$) there exists a $(1+\varepsilon)$ bi-Lipschitz map (respectively $R$-bounded rescaling undistorted map) $f: Q_k \to Q_N$ such that $f(Q_k) \subset \Dcal$.
It is not immediately clear that these numbers are finite (see Remark \ref{usingDensityHalesJewett}).
Regarding the above numbers, we prove the following:\footnote{Recall that for two sequences $\{a_n\}_n, \{b_n\}_n \subset \real^+$ one has $a_n = O(b_n)$ $\iff$ $a_n \lesssim b_n$ $\iff$ $b_n = \Omega(a_n)$ $\iff$ $b_n \gtrsim a_n$ if and only if there exists $0<C<\infty$ such that $a_n \leq C b_n$ for all $n \in \n$. Also, we write $a_n \asymp b_n$ when $a_n \lesssim b_n$ and $b_n \lesssim a_n$. The same notation is used whenever we have a function of several parameters.}

\begin{thmB}[Main theorem]\label{thm:MAIN}
    For every $k \in \n,\;0 < \delta < 1,\; 0 < \varepsilon < 1/4,\;R \geq 2$,
    \begin{equation}\label{eq:MAIN_biLipschitz_bound}
        k + \log(1/\delta) \leq \Emb(1+\varepsilon,\delta,k) \leq O\left(k^3\dfrac{1}{\varepsilon^2}\log\Bigl(\frac{1}{\delta}\Bigr)\right)
    \end{equation}
    \begin{equation}
        e^{\Omega(k)}\,\log\bigl(\tfrac{1}{\delta}\bigr) \leq \Emb(1,\delta,k) \leq \exp \left({e^{O(k)} \log\left(\frac{1}{\delta}\right)}\right)
    \end{equation}
    \begin{equation}
        \exp\left({e^{\Omega_R(k)}\;\log\left(\frac{1}{\delta}\right)}\right) \leq \Emb^{\leq R}(1,\delta,k) \leq \exp \left({e^{O(k)} \log\left(\frac{1}{\delta}\right)}\right).
    \end{equation}
\end{thmB}

Note that the lower bound on $\Emb(1+\varepsilon,\delta,k)$ is trivial (any $(1+\varepsilon)$-distorted copy of $Q_k$ has exactly $2^k$ points).
\textbf{We conjecture that $\Emb(1,\delta,k) \leq e^{O(k)} \log \frac{1}{\delta}$ for all $k \in \n$ and $0 < \delta < 1$.}
See Section \ref{sec:conclusion} for questions and further directions.

There are countless Ramsey theorems which find \textit{combinatorial copies} of cubes inside dense subsets of cubes.
We do not attempt to survey them here.
Two landmark results of this type are the Sauer--Shelah lemma \cite{sauer1972density,shelah1972combinatorial} and the multi-dimensional density Hales--Jewett theorem of Furstenberg--Katznelson \cite{FurstenbergKatznelson1991} (see the discussion on Section \ref{sec:remarksProofs}).
The copies of cubes in these theorems need not preserve the Hamming metric in any meaningful way.
We mention the work of Benjamini--Cohen--Shinkar \cite{benjamini2016bi} who construct a distortion $20$ embedding
of the $k$-dimensional Hamming cube $Q_k$ into the ball of radius $\lfloor
k/2\rfloor+1$ in the Hamming cube $\{0,1\}^{k+1}$ (which is a subset of density $>1/2$)
answering a question of Lovett--Viola \cite{lovett2011bounded}.
Finally, we remark that the density Hales--Jewett theorem of Furstenberg--Katznelson \cite{FurstenbergKatznelson1991} 
gives a short proof that $\Emb(1,\delta,k) < \infty$ for all $k \in \n$ and $0<\delta < 1$ 
without yielding any reasonable quantitative bounds (see Remark \ref{usingDensityHalesJewett} in Section \ref{sec:remarksProofs}).

\subsection{Metric embeddings of paths into dense subsets of paths}\label{subSect-relaxedSzemeredi}
\;

We consider the \textbf{path space} $[N] = \{1,...,N\}$ whose metric is given by $d(i,j):=|i-j|$ and endow it with the uniform probability measure, i.e. $\mu(\Dcal) := |\Dcal|/N$ for all $\Dcal \subset [N]$. Rödl--Sales \cite{rodl2022blurred} prove that for any $\varepsilon>0, r \in \n, k \in \n$, and for any $r$-coloring of the path of length $N = O_{r, \varepsilon}(k^r)$, there is a $(1+\varepsilon)$-distorted embedding $f: [k] \to [N]$ such that $f([k])$ is monochromatic.
Conversely, they show that there exists an $r$-coloring of $[\Omega_{r, \varepsilon}(k^r)]$ which contains no monochromatic $(1+\varepsilon)$-bi-Lipschitz copy of $[k]$.
We prove a density version of this result, also giving a tight dependence on $k$.

Fix $k \in \n, 0 < \delta < 1, \varepsilon \geq 0$.
\underline{\textbf{We denote by $Path(1+\varepsilon,\delta,k)$}} the smallest $N$ such that for every subset $\Dcal \subset [N]$ with $\mu(\Dcal) \geq \delta$ (i.e. $|\Dcal| \geq \delta N$) there exists a $(1+\varepsilon)$ bi-Lipschitz map $f: [k] \to [N]$ such that $f([k]) \subset D$.
Bounds on $Path(1,\delta,k)$ correspond exactly to bounds on Szemerédi's theorem, so the following theorem can be viewed as a bi-Lipschitz relaxation of Szemerédi's theorem:

\begin{thmB}[metric embeddings of paths into dense subsets of paths]\label{thm:pathSpaces}
    \;\\For all $k\geq 3, 0 < \delta < 1, 0 < \varepsilon \leq 1$, we have:
    $$e^{\Omega(k \log(1/\delta))} \leq Path(1+\varepsilon,\delta,k) \leq e^{O( k \varepsilon^{-1} \log(1/\delta))}.$$
\end{thmB}

The upper bound is a strengthening of Theorem 6 of Dumitrescu \cite{dumitrescu2010approximate} who showed $Path(1+\varepsilon,\delta,k)  = e^{O_\varepsilon(k \log k \log(1/\delta))}$ (see also Proposition 3.7 of Rödl--Sales \cite{rodl2022blurred} for a different proof and  \cite{fraser2018arithmetic,fraser2021approximate,rodl2022blurred} for related results).

\subsection{Metric embeddings of trees into dense subsets of trees}\label{subSect-relaxedFurstenberg-Weiss}
\;

For each $N \in \n$ consider the graph of the complete binary tree of depth $N$. We will denote it by $Tree(N)$, and represent each vertex as a string $w$ of $0$'s and $1$'s of length $l(w) \in \{0,...,N-1\}$. 
The shortest path metric on the vertices of this graph is a \textbf{tree metric} which we denote by $d_T$.
Following \cite{furstenberg2003markov, pach2011remarks}, we define a probability measure on the vertices of $Tree(N)$ by
$$\mu(\Dcal):= \dfrac{1}{N}\sum_{w \in \Dcal} 2^{-l(w)}\;\;\;\text{for all }\Dcal \subset Tree(N).$$
Fix $k \in \n, 0 < \delta < 1, \varepsilon \geq 0$. \underline{\textbf{We denote by $Tree(1+\varepsilon,\delta,k)$}} the smallest $N$ such that for every subset $\Dcal \subset Tree(N)$ with $\mu(\Dcal) \geq \delta$ there exists a $(1+\varepsilon)$ bi-Lipschitz map $f: Tree(k) \to Tree(N)$ with $f(Tree(N)) \subset D$.
The Furstenberg--Weiss theorem \cite{furstenberg2003markov} states that $Tree(1, \delta, k)<\infty$ for all $k \in \n, 0 < \delta < 1$, while Pach--Solymosi--Tardos \cite{pach2011remarks} give quantitative bounds on $Tree(1, \delta, k)$ using the bounds for Szemerédi's theorem.
The following theorem concerns the $(1+\varepsilon)$-distorted setting:

\begin{thmB}[Metric embeddings of trees into dense subsets of trees]\label{thm:treeRamsey}
\;\\For all $k\geq 3,\; 0 < \delta < 1,\; 0 < \varepsilon <0.001$, we have:
\[e^{\Omega(k \log(1/\delta))} \leq Tree(1+\varepsilon,\delta,k) \leq e^{O( k \varepsilon^{-1} \log(1/\delta))}.\]
\end{thmB}

\subsection{Large subsets of the Hamming cube embed poorly into spaces of non-positive curvature}
The subsequent geometric application of Theorem \ref{thm:MAIN} was proposed to us by A.~Naor \cite{naor2022personal}. Following \cite{BartalLinialMendelNaor2005}, given a class of metric spaces $\Mcal$, a finite metric space $(X,d)$, and $\alpha>1$ write 
\[R_{\Mcal}(X,\alpha) := \sup\{|\Dcal| : \Dcal \subset X\;\;\text{and}\;\; \exists M \in \Mcal \;\;\text{s.t.}\;\; c_M(\Dcal) \leq \alpha\}\]
where $c_M(\Dcal)$ denotes the infimal bi-Lipschitz distortion over all functions $f: \Dcal \to M$.
See \cite{bartal2006ramsey,BartalLinialMendelNaor2005,mendel2007ramsey,mendel2013ultrametricHausdorff,mendel2013ultrametric} for motivations for studying this parameter from the perspective of geometry (nonlinear Dvoretzky theorems, geometric measure theory, probability) and computer science (online algorithms, data structures).

Let $\mathrm{POS}$ denote the class of metric spaces of \textbf{non-negative Alexandrov curvature},  and let $\mathrm{CAT}(0)$ denote the class of metric spaces of \textbf{non-positive Alexandrov curvature} (the relevant standard definitions can be found for example in the monograph \cite{bridson2013metric}).
Bartal--Linial--Mendel--Naor \cite{BartalLinialMendelNaor2005} showed that
\begin{equation}\label{eq:quote BLMN}
    2^{N(1-\frac{C \log \alpha}{\alpha^2})} \leq R_{\mathrm{POS}}(Q_N,\alpha)\leq C 2^{N(1-\frac{c}{\alpha^2})}\;\;\;\text{for all }\alpha>1,N \in \n.
\end{equation}
where $C,c>0$ are universal constants.
The works of Gromov \cite{gromov2001cat,gromov2003random} and Kondo \cite{kondo2012cat} imply that the method of \cite{BartalLinialMendelNaor2005} for the upper bound fails when $\mathrm{POS}$ is replaced by $\mathrm{CAT}(0)$ , and Naor asked if one could prove an upper bound on the size of subsets of $Q_N$ which embed into $\mathrm{CAT}(0)$ spaces.
In Section \ref{sec:BMWTypeApplication} we prove
\begin{equation}\label{eq:RamseyForNonPositiveCurvature}
    R_{\mathrm{CAT}(0)}(Q_N,\alpha) \leq 2^{N(1-C/\alpha^4)}\;\;\;\text{for all }\alpha>1,\;\text{and}\;N\geq c \alpha^6,
\end{equation}
where $C,c>0$ are universal constants.

More generally, the upper bound in \cite{BartalLinialMendelNaor2005} restricts the embeddability into any space of nontrivial \textbf{Markov type}
\footnote{The notion of \textbf{Markov type} was introduced by Ball in~\cite{ball1992markov}. We will also use below the classical notion of \textbf{BMW type}
which is a variant of Enflo type~\cite{enflo1978infinite} 
introduced by Bourgain, Milman and Wolfson in~\cite{bourgain1986type}.
Both BMW type and Markov type are important  bi-Lipschitz invariants with multifaceted applications, but we do not need to recall their definitions here; see e.~g.~the survey~\cite{naor2012introduction}.
Theorem \ref{thm:applicationToBMWType} also works with the Enflo type constant instead of the BMW type constant, but we do not know how to prove the analogous statement of Proposition \ref{prop::union_of_subspaces_and_metric_type} for Enflo type.
}.
This implies the second inequality in \eqref{eq:quote BLMN} because Ohta proved \cite{ohta2009markov} that every metric space in $\mathrm{POS}$ has Markov type $2$. 
On the other hand, Gromov \cite{gromov2001cat,gromov2003random} and Kondo \cite{kondo2012cat} constructed metric spaces in $\mathrm{CAT}(0)$ that fail to have any nontrivial Markov type, hence the method of \cite{BartalLinialMendelNaor2005} does not apply to the setting of \eqref{eq:RamseyForNonPositiveCurvature}.
Here we prove:

\begin{thmB}[Large subsets of the Hamming cube cannot embed into spaces with BMW type]\label{thm:applicationToBMWType}
    \;\\There are universal constants $c,c'>0$ such that 
    \\ for any $\alpha>1$, $N \geq c \alpha^6$, $p>1$, and any metric space $(M,d)$, 
    \\\textbf{If} $\Dcal \subset Q_N$ admits $f: \Dcal \to (M,d)$ with bi-Lipschitz distortion $\leq\alpha$
    \\\textbf{then} the size of $\Dcal$ is small in the following sense:
    \[\log_2 |\Dcal|\leq N\left(1-\frac{c'}{(T_p(M)\;\alpha)^{2p/(p-1)}}\right)\]
    where $T_p(M)$ denotes the BMW type $p$ constant of $(M,d)$.
\end{thmB}
Inequality (\ref{eq:RamseyForNonPositiveCurvature}) follows from Theorem \ref{thm:applicationToBMWType} and the fact that any metric space $(M,d)$ of non-positive Alexandrov curvature has BMW type $2$ constant $T_2(M)=1$ (the proof essentially goes back to Enflo \cite{enflo1970nonexistence}, see e.g. \cite{ohta2009markov}).

\subsection{Union problems for metric type and universality}
\;

Next, we describe some applications of the Hales--Jewett theorem in metric geometry.
(The applications are essentially qualitative, so we do not need the fine quantitative bounds of Theorem \ref{thm:MAIN}.)
\\\textbf{Union problems} in metric geometry are problems of the following form:

\textit{If a metric space $(M,d)$ is expressed as the union of two subspaces $A,B \subset M$, i.e. $M = A \cup B$, and both $A$ and $B$ have a certain metric property, then does $M$ also have this metric property?}

The union problem has been answered affirmatively when the metric property is 
coarse embeddability into Hilbert space \cite{DadarlatGuentner+2007+1+15}, 
bi-Lipschitz embeddability into ultrametrics \cite{mendel2013ultrametric},
bi-Lipschitz embeddability into Hilbert space \cite{MakarychevMakarychev},
asymptotic dimension \cite{bell2001asymptotic}, 
and finite decomposition complexity \cite{guentner2013discrete}.
\footnote{The union problem for bi-Lipschitz embeddability into $L_1$ remains a well-known open problem.}
See also \cite{OstrovskiiRandrianantoanina+2022+313+329} for related results and a  discussion on union problems.

Given a metric space $(M,d)$, denote by $p_M$ the supremal $p\geq 1$ such that $M$ has BMW type $p$.
Whenever $p_M>1$ we say that $M$ has \textbf{nontrivial BMW type}.

\begin{propB}[The union problem for nontrivial BMW type] \label{prop::union_of_subspaces_and_metric_type}
    \;\\\textbf{If} $(M,d)$ is a metric space with $M = A \cup B$ for some $A,B \subset M$
    \textbf{then}
    $$p_{A}>1\;\;\text{and}\;\;p_{B}>1 \implies p_M > 1.$$
\end{propB}

By an iterative application of Proposition \ref{prop::union_of_subspaces_and_metric_type}, the same result holds for finite unions.

Using BMW type, we also address the following natural question:
\textit{Can we decompose every metric space into a non-positive curvature subspace and a non-negative curvature subspace?}
The answer is no, even when we allow bi-Lipschitz perturbations of each of the two pieces.

\begin{propB}[Failure to decompose into non-positive and non-negative curvature]\label{prop:decomposition_into_positive_and_negative}
    \;\\For any $D>1$, there exists a metric space $(M,d)$ (namely the Hamming cube $Q_N$ for some $N \in \n$) which cannot split as the union of subsets $M = A \cup B$ such that 
    \\--$A$ embeds into a metric space of non-negative Alexandrov curvature with distortion $\leq D$ and
    \\--$B$ embeds into a metric space of non-positive Alexandrov curvature with distortion $\leq D$.
\end{propB}

For the next result, we call a metric space $M$ \textbf{universal} if for any $\varepsilon>0$, every finite metric space embeds into $M$ with bi-Lipschitz distortion $< 1 + \varepsilon$.
Note that this version of universality coincides with  trivial metric cotype\footnote{The notion of \textbf{metric cotype} was introduced by M. Mendel and A. Naor \cite{mendel2008metric} and is an important bi-Lipschitz invariant.
We do not need to recall the definition here, because we will not use it; see e.~g.~the survey~\cite{naor2012introduction}.}.

\begin{propB}[A universal metric space cannot be split into two non-universal subsets]\label{prop::union_of_subspaces_and_metric_cotype}
    \;\\\textbf{If} $(M,d)$ is a metric space with $M = A \cup B$ for some $A,B \subset M$
    and $M$ is universal,
    \\\textbf{then} either $A$ is universal or else $B$ is universal.
\end{propB}

The proofs of Propositions \ref{prop::union_of_subspaces_and_metric_type}, \ref{prop:decomposition_into_positive_and_negative}, and \ref{prop::union_of_subspaces_and_metric_cotype} are in Section \ref{sec:remarksProofs}.
Propositions \ref{prop::union_of_subspaces_and_metric_type} and \ref{prop::union_of_subspaces_and_metric_cotype} are examples of the following (vague) general statement:
if a metric property admits a test space characterization (see Chapter 9 in  \cite{ostrovskii2013metric} for the definition), and there exists a $2$-coloring Ramsey theorem for this family of test spaces (e.g. the $2$-coloring counterpart of Question \ref{question:vague_direction}), then this metric property is preserved by taking finite unions of subsets.

\subsection{Paper Organization}
Section \ref{sec:preliminaries} sets up some general notation for the proofs that follow.
In Section \ref{sec:UpperBoundBoundedRescaling} we prove the upper bound for undistorted Hamming cubes of bounded rescaling.
In Section \ref{sec:UpperBoundEpsilonDistorted} we prove the upper bound for $(1+\varepsilon)$-distorted Hamming cubes.
In Section \ref{sec:lowerBounds} we prove the lower bounds for embedding Hamming cubes.
In Section \ref{sec:pathSpaces} we discuss metric embeddings of path spaces, and in Section \ref{sec:treesRamsey} we discuss metric embeddings of tree metrics.
In Section \ref{sec:BMWTypeApplication} we prove the application regarding $CAT(0)$ spaces.
In Section \ref{sec:remarksProofs} we discuss union problems and prove remarks made in the introduction.
Finally, in Section \ref{sec:conclusion}, we list some questions and directions for further study.

\subsection*{Acknowledgements}
We thank Pandelis Dodos for several fruitful discussions regarding the problem.
We thank Assaf Naor for the application regarding non-positive curvature and its context.
Cosmas Kravaris thanks the Hausdorff Research Institute for Mathematics for their hospitality and giving him the opportunity to meet Pandelis Dodos during the Workshop on Analysis and Geometry on Discrete Spaces on October 7-11, 2024.

\section{\textbf{General notation and conventions}}\label{sec:preliminaries}

We denote by $Q_N=\{0,1\}^N$ the $N$-dimensional hypercube equipped with the Hamming metric
\[ d_N(x,y) = \norm{x-y}_1 = \abs*{\Set*{i\in[N]: x_i\neq y_i}},\;\;\;\text{for all }x,y\in Q_N.\]
Given any set $X$ by $\abs{X}$ we denote its cardinality. For every $n\in\mathbb{N}$, we set $[n]\coloneqq \{1,\dots,n\}$, and  we will identify subsets of $[n]$ with vectors in $ Q_n$ which we equip with the Hamming distance, i.e. for $x,y\in Q_n$ their distance is given by $\norm{x-y}=\abs*{x\bigtriangleup y}$. Given any non-negative integer $k$ we define $\binom{[n]}{k}\coloneqq\{x\in Q_n: |x|=k\}$; moreover, for every $t\meg 0$, we set $[n]^{\mik t}\coloneqq \{x\in Q_n: |x|\mik t\}$, and similarly we define the sets $[n]^{<t}, [n]^{\meg t}, [n]^{>t}$. 

Throughout the paper the symbol $\mu$ denotes the uniform probability measure corresponding to the ambient space under consideration which will always be clear from context, with the binary tree $Tree_N$ on $N$ levels being the only exception (the measure $\mu$ for the tree was given in subsection \ref{subSect-relaxedFurstenberg-Weiss}). Finally we write $\log$ for logarithm in base $2$, and we use $\ln$ for the natural logarithm, while the symbols $\prob$ and $\ave$ denote probability and expectation with respect to auxiliary probability spaces.

Recall that the \textbf{bi-Lipschitz distortion} of an injective map between two metric spaces $f: (M_1,d_1) \to (M_2,d_2)$ is defined to be:
\[distortion(f) := \left(\sup_{x_1\neq x_2, x_1,x_2 \in M_1} \dfrac{d_1(x_1,x_2)}{d_2(f(x_1),f(x_2))}\right) \cdot \left(\sup_{x_1\neq x_2, x_1,x_2 \in M_1} \dfrac{d_2(f(x_1),f(x_2))}{d_1(x_1,x_2)}\right).\]
We write $c_{M_2}(M_1)$ for the infimal distortion over all injective functions $f: M_1 \to M_2$.
\\Whenever $M_2 = L_2$ is the Hilbert space, we write $c_2(M_1):= c_{M_2}(M_1)$.

\section{\textbf{Upper bound for the bounded rescaling undistorted copies of cubes}}\label{sec:UpperBoundBoundedRescaling}

The map which we find and sends $Q_k$ inside $\Dcal \subset Q_N$ has a very special form.
Suppose $N = nk$ and write $Q_N = \prod_{i=1}^kQ_n$ meaning that each string $x \in Q_N$ is expressed as the concatenation of $k$ strings of length $n$, $x = (x^{1},x^2,...,x^k)$.
Call the image of an embedding $f: Q_k \to \{0,1\}^{nk}$ an \textbf{equitable block copy} whenever there exist $x_0^1,x_1^1,x_0^2,x_1^2,...,x_0^k,x_1^k \in Q_n$ such that 
$$||x^1_0 - x^1_1||_1 = ||x^2_0 - x^2_1||_1 = \cdots =||x^k_0 - x^k_1||_1$$
and the embedding $f$ is given by: $f(\alpha) = (x_{\alpha_1}^1, x_{\alpha_2}^2, ..., x_{\alpha_k}^k)$ for all $\alpha \in Q_k$.
It is clear that an equitable block copy is a rescaled isometric embedding of $Q_k$ into $Q_N$ with rescaling $r:=||x^1_0 - x^1_1||_1$. We will need the following definition.

\begin{defnB}\label{defn:section_on_element}
Let $n,N\in\mathbb{N}$ and $\Dcal\subset Q_{n+N}$. For every $x\in Q_n$ the \textbf{section} of $\Dcal$ on $x$ is the set:
\begin{equation}\label{eq:section_on_element}
\Dcal_{x}\coloneqq\Set{y\in Q_N:(x,y)\in\Dcal}.
\end{equation}
\end{defnB} 

In order to find the pairs of strings $(x_0^1,x_1^1),(x_0^2,x_1^2),...,(x_0^k,x_1^k)$ we use induction. The goal of the base case of the induction is to find a pair $(x_0^1,x_1^1) \in Q_n \times Q_n$ such that $||x_0^1-x_1^1||_1 = r$ and the density $\mu(\Dcal_{x_0^1} \cap \Dcal_{x_1^1})$ is as large as possible. Lemma \ref{lem:existence_of_pair_with_big_intersection} gives sufficient conditions for a family of events in a probability space to contain a pair whose intersection has  high probability. In Lemma \ref{lem:for_the_inductive_step} which is the heart of the proof, we first find a layer of $Q_n$ near the middle for which the expectation of the size of the sections is large. Then for each subset of $[n]$ in the layer just below, we consider all of its one-element extensions, and by a double counting argument, we find a subset where the expectation over all of its one-element extensions is also high. This allows us to apply Lemma \ref{lem:existence_of_pair_with_big_intersection} and achieve the goal mentioned above. The rest of the proof is a straightforward induction scheme. 

\begin{lemB}\label{lem:existence_of_pair_with_big_intersection}
Let $(A_i)_{i\in[m]}$ be events in a probability space $(\Omega,\Fcal,\mu)$ with $\frac{1}{m}\sum_{i=1}^m \mu(A_i)\meg \delta$.
\\\textbf{If} $m> 2/\delta$, \textbf{then} there exist distinct $i,j$ such that
\[
\mu(A_i\cap A_j) > \frac{\delta^{2}}{2}.
\]
\end{lemB}
\begin{proof}
From Jensen's inequality,
\[
\mathbb{E}\Bigl[\Bigl(\sum_{i=1}^m \mathbf{1}_{A_i}\Bigr)^{2}\Bigr]
   \meg \Bigl(\mathbb{E}\sum_{i=1}^m \mathbf{1}_{A_i}\Bigr)^{2}
   = \Bigl(\sum_{i=1}^m\mu(A_i)\Bigr)^{2}
   \meg (m\delta)^{2}.
\]
Expanding the square,
\[
\mathbb{E}\Bigl[\Bigl(\sum_{i=1}^m \mathbf{1}_{A_i}\Bigr)^{2}\Bigr]
   = \sum_{i=1}^m\mu(A_i) + \sum_{i\neq j}\mu(A_i\cap A_j).
\]
Hence
\[
\sum_{i\neq j}\mu(A_i\cap A_j)
   \meg \left(\sum_{i=1}^m\mu(A_i)\right)^{2} - \sum_{i=1}^m\mu(A_i)
   \meg m\delta(m\delta - 1).
\]
The average over the $m(m-1)$ unordered pairs is therefore
\[
\frac{1}{m(m-1)}\sum_{i\neq j}\mu(A_i\cap A_j)
   \meg \frac{m\delta(m\delta - 1)}{m(m-1)}
   = \delta^{2}\,\frac{m\delta-1}{m\delta-\delta}.
\]
The function $g(x)=(x-1)/(x-\delta)$ is increasing for $x>\delta$.  Since $m\delta>2$ by hypothesis, $g(m\delta)\meg g(2)=(2-1)/(2-\delta)>1/2$.  Consequently the average exceeds $\delta^{2}/2$, so some specific pair must satisfy $\mu(A_i\cap A_j)>\delta^{2}/2$.
\end{proof}


\begin{lemB}[the inductive step]\label{lem:for_the_inductive_step}
Let $n,N\in\mathbb{N}$, $\delta\in(0,1)$ and $\Dcal\subseteq Q_n\times Q_N$ with $\mu(\Dcal)\meg\delta$\;\\\textbf{If} we have
\[
n \meg \max\Bigl\{8\ln\!\Bigl(\frac{4}{\delta}\Bigr),\;\frac{16}{\delta}\Bigr\},
\]
\textbf{then} there exist $x_0,x_1\in Q_n$ with $\norm{x_0-x_1}_1=2$ such that
\[
\mu\bigl(\Dcal_{x_0}\cap\Dcal_{x_1}\bigr) \meg \frac{\delta^{2}}{8}.
\]
\end{lemB}
\begin{proof}
Let $\mu$ denote uniform measure on $ Q_n$ as well.  By Hoeffding's bound,
\[
\mu\Bigl(\bigl\{x: \bigl|\abs{x}-\tfrac{n}{2}\bigr|\meg \tfrac{n}{4}\bigr\}\Bigr)
   = 2\tailsize{1/2}{n}
   \mik 2\exp\!\Bigl(-\frac{n}{8}\Bigr).
\]
Our assumption on $n$ guarantees $2\exp(-n/8)\mik \delta/2$.  Therefore the average section density on the ``middle strip'' $\{x:|\abs{x}-n/2|<n/4\}$ satisfies
\[
\ave_{|\abs{x}-n/2|<n/4}\mu(\Dcal_x)
   \meg \frac{\delta-(\delta/2)\cdot1}{1-\delta/2}
   > \frac{\delta}{2}.
\]
Consequently there exists a layer $\binom{[n]}{j}$ (with $j$ between $n/4$ and $3n/4$) such that
\[
\ave_{x\in\binom{[n]}{j}}\mu(\Dcal_x) > \frac{\delta}{2}.
\]

Fix such a $j$.  For each $(j-1)$-set $y$ let
\[
\Ccal_y\coloneqq\bigl\{x\in\tbinom{[n]}{j}: y\subset x\bigr\}.
\]
Every $x\in\binom{[n]}{j}$ contains exactly $j$ different $(j-1)$-subsets, so
\[
\sum_{x\in\binom{[n]}{j}}\mu(\Dcal_x)
   = \frac{1}{j}\sum_{y\in\binom{[n]}{j-1}}
        \sum_{x\in\Ccal_y}\mu(\Dcal_x).
\]
Dividing by $\binom{n}{j}$ gives
\[
\ave_{x\in\binom{[n]}{j}}\mu(\Dcal_x)
   = \ave_{y\in\binom{[n]}{j-1}}
        \ave_{x\in\Ccal_y}\mu(\Dcal_x).
\]
Hence there exists a $(j-1)$-set $y$ for which
\[
\ave_{x\in\Ccal_y}\mu(\Dcal_x) > \frac{\delta}{2}.
\]

Now $\Ccal_y$ consists of all $j$-sets containing $y$; its size is $n-j+1\ge n/4$.  Moreover, any two distinct elements $x,x'\in\Ccal_y$ satisfy $\norm{x-x'}=2$ (they differ in exactly one element outside $y$).  Because $n/4 > 2/(\delta/2)=4/\delta$ by the hypothesis on $n$, Lemma~\ref{lem:existence_of_pair_with_big_intersection} applied to the family $\{\Dcal_x: x\in\Ccal_y\}$ with density $\delta/2$ yields $x_0,x_1\in\Ccal_y$ with
\[
\mu(\Dcal_{x_0}\cap\Dcal_{x_1}) > \frac{(\delta/2)^{2}}{2}= \frac{\delta^{2}}{8},
\]
and $\norm{x_0-x_1}=2$ as required.
\end{proof}

We now iterate Lemma~\ref{lem:for_the_inductive_step} to build an undistorted copy of $Q_k$ in any large enough $\Dcal$.

\begin{thmB}\label{thm:upper_bound_for_Emb(1,delta,k)}
Define $\gamma(\delta)\coloneqq \max\Bigl\{8\ln\Bigl(\frac{4}{\delta}\Bigr),\;\frac{16}{\delta}\Bigr\}$.  For $k\ge1$ and $\delta\in(0,1)$ set
\[
N(k,\delta)\coloneqq 2+\sum_{i=1}^{k}
                     \gamma\!\Bigl(\frac{\delta^{2^{i-1}}}{8^{\,2^{i-1}-1}}\Bigr).
\]
Every $\Dcal\subseteq Q_{N(k,\delta)}$ with $\mu(\Dcal)>\delta$ contains an isometric copy of $Q_k$ with rescaling $2$.  Consequently,
\[
\Emb(1,\delta,k) \mik N(k,\delta)
                \mik 2^{\,2^{k}\log(8/\delta)}.
\]
\end{thmB}
\begin{proof}
We prove by induction on $k$ that if $N\meg N(k,\delta)$ and $\mu(\Dcal)>\delta$, then $\Dcal$ contains an isometric copy of $Q_k$ with rescaling $2$.

\paragraph{Base $k=1$.}  Take $n_1=\gamma(\delta)$.  Since $N(k,\delta)\meg n_1+2$, we can write $Q_N = Q_{n_1}\times Q_{N-n_1}$ with $N-n_1\meg2$.  Apply Lemma~\ref{lem:for_the_inductive_step} to $\Dcal$ viewed as a subset of $Q_{n_1}\times Q_{N-n_1}$; the lemma gives $x_0,x_1\in Q_{n_1}$ with $\norm{x_0-x_1}=2$ and $\mu(\Dcal_{x_0}\cap\Dcal_{x_1})>\delta^{2}/8>0$.  Hence there exists $y\in Q_{N-n_1}$ such that $(x_0,y),(x_1,y)\in\Dcal$.  Mapping $0\mapsto (x_0,y)$, $1\mapsto (x_1,y)$ provides the desired isometric embedding of $Q_1$.

\paragraph{Inductive step.}  Assume the statement holds for $k-1$.  Set $n_1=\gamma(\delta)$ and note that
\[
N(k,\delta)-n_1 = 2+\sum_{i=2}^{k}\gamma\!\Bigl(\frac{\delta^{2^{i-1}}}{8^{\,2^{i-1}-1}}\Bigr)
                = N\!\Bigl(k-1,\frac{\delta^{2}}{8}\Bigr).
\]
Write $Q_N = Q_{n_1}\times Q_{N-n_1}$ with $N\meg N(k,\delta)$.  By Lemma~\ref{lem:for_the_inductive_step} applied to $\Dcal\subseteq Q_{n_1}\times Q_{N-n_1}$, there exist $x_0,x_1\in Q_{n_1}$ with $\norm{x_0-x_1}=2$ and
\[
\mu\bigl(\Dcal_{x_0}\cap\Dcal_{x_1}\bigr) > \frac{\delta^{2}}{8}.
\]
Define $\widetilde{\Dcal}\coloneqq \Dcal_{x_0}\cap\Dcal_{x_1}\subseteq Q_{N-n_1}$.  Since $N-n_1\meg N(k-1,\delta^{2}/8)$, the inductive hypothesis applied to $\widetilde{\Dcal}$ with density parameter $\delta^{2}/8$ yields an isometric embedding $g:Q_{k-1}\to\widetilde{\Dcal}$ satisfying $\norm{g(\beta)-g(\beta')}=2\norm{\beta-\beta'}$.

Now define $f:Q_k\to\Dcal$ by
\[
f(\alpha_1,\beta) \coloneqq 
\begin{cases}
(x_0,g(\beta)) & \text{if }\alpha_1=0,\\[2pt]
(x_1,g(\beta)) & \text{if }\alpha_1=1,
\end{cases}
\qquad (\beta\in Q_{k-1}).
\]
Both $(x_0,g(\beta))$ and $(x_1,g(\beta))$ belong to $\Dcal$ because $g(\beta)\in\widetilde{\Dcal}$.  For two vectors $\alpha=(\alpha_1,\beta)$, $\alpha'=(\alpha_1',\beta')$ we have
\[||f(\alpha)-f(\alpha')||
   = ||x_{\alpha_1}-x_{\alpha_1'}|| + ||g(\beta)-g(\beta')||
   = 2|\alpha_1-\alpha_1'| + 2||\beta-\beta'||
   = 2||\alpha-\alpha'||,\]
Hence $\norm{f(\alpha)-f(\alpha')}=2\norm{\alpha-\alpha'}$, so $f$ is an isometric embedding of $Q_k$ with rescaling $2$.

Finally, we get a bound on $N(k,\delta)$.
For $\delta\le1$ we have $\gamma(\delta)\mik 16/\delta+8\ln(4/\delta)\mik 32/\delta$.  Then
\[
N(k,\delta)\mik
2+32\sum_{i=1}^{k}\frac{8^{\,2^{i-1}-1}}{\delta^{\,2^{i-1}}}
=2+4\sum_{i=1}^{k}\left(\frac{8}{\delta}\right)^{2^{i-1}}
\mik 2+\frac{32}{7}\left(\frac{8}{\delta}\right)^{2^{k-1}}
\mik \left(\frac{8}{\delta}\right)^{2^{k}}
=2^{\,2^{k}\log(8/\delta)}.
\]
This gives the claimed double‑exponential upper bound.
\end{proof}

\section{\textbf{Upper bound for $(1+\varepsilon)$-distorted copies of cubes}}\label{sec:UpperBoundEpsilonDistorted}

The key difference between $(1+\varepsilon)$-distorted and undistorted copies of Hamming cubes is that the former can be infinitesimally perturbed and still maintain the same geometric structure. The proof for the upper bound on $\Emb(1+\varepsilon, \delta, k)$ has two stages. The first stage follows an induction scheme similar to the one we used for the undistorted copies but this time we exploit the flexibility we mentioned earlier to find a first exponential upper bound that guarantees \textbf{roughly equitable block copies} which are given by $x_0^1,x_1^1,x_0^2,x_1^2,...,x_0^k,x_1^k \in Q_n$ satisfying: 
\[(1-\varepsilon_1)\dfrac{n}{2} \leq ||x^i_0 - x^i_1||_1 \leq (1+\varepsilon_1)\dfrac{n}{2}\;\;\;\text{for each }i=1,...,k.\]

This is the content of Proposition \ref{prop:block_k_subspaces}. We call the parameter $\varepsilon_1$ \textbf{block error}. In the second stage we show that given a \textbf{perturbation error} $\varepsilon_2$, if we can find a roughly equitable block copy with block error $\varepsilon_1$ inside the $\varepsilon_2$-neighborhood of $\Dcal$ (for the definition see below), then we can find an $\varepsilon=\Theta(\varepsilon_1+k\varepsilon_2)$ distorted copy in $\Dcal$. This is the content of Proposition \ref{prop:subspace_in_Dε_then_subspace_in_D_two_eps}. The proof is completed by combining these two results.  

\subsection{More notation and some basic facts about concentration of measure}\label{subSec:concentrationMeasureFacts}
\begin{defn}\label{defn:tailsize}
Let $n\in\mathbb{N}$, $\varepsilon\in[0,1]$ and $\mu$ the uniform probability measure on $Q_n$. We define:
\begin{equation}\label{eq:tailsize}
\tailsize{\varepsilon}{n}\coloneqq \mu\left([n]^{\mik (1-\varepsilon)\frac{n}{2}}\right)=\mu\left([n]^{\meg (1+\varepsilon)\frac{n}{2}}\right).
\end{equation}
\end{defn}

We will use the following rough bound for the tails of the binomial distribution: 
\begin{lemB}[Hoeffding bounds]\label{lem:Hoeffding_bounds_for_tails}
For every $n\in\mathbb{N}$ and $\varepsilon\in[0,1]$ it holds that  
\begin{equation}\label{eq:Hoeffding_bounds_for_uniform_tails}
\tailsize{\varepsilon}{n}\mik \exp\left(-\frac{\varepsilon^2 n}{2}\right).
\end{equation}
\end{lemB}
$\mbox{}$
\begin{defnB}\label{defn:epsilon_neighborhood_of_set}
Let $n\in\mathbb{N}$, $\Dcal\subseteq  Q_n$ and $\varepsilon\in[0,1]$. The \textbf{$\varepsilon$-neighborhood} of $\Dcal$ is the set:
\begin{equation}\label{eq:epsilon_neighborhood_of_set}
\Dcal_{\varepsilon}\coloneqq \{x\in Q_n: \exists x'\in\Dcal \text{ such that } \norm{x-x'}\mik \varepsilon n\}.
\end{equation}
\end{defnB}

\begin{lemB}[Harper's isoperimetric inequality \cite{HARPER1966385}]\label{lemm:HarperIsoperimetry}
    Let $N \in \n$ and let $Q_N$ be the Hamming cube of dimension $N$.
    Denote by $Ball(r) \subset Q_N$ the ball of radius $r\in \n$  centered at $\vec{0}$ (which consists of all indicators of subsets of $[N]$ of size $\leq r$).
    For any subset $\Acal \subset Q_N$ we have
    $$|\Acal| \geq |Ball(r)| \implies |\Acal_{\varepsilon}| > |Ball(r+ \varepsilon N)|\;\;\;\text{for all }\Acal \subset Q_N, N \in \n, 0<\varepsilon<1.$$
\end{lemB} 

A standard application of Harper's inequality is the following estimate:

\begin{lemB}[Concentration of measure phenomenon for the Hamming cube]\label{lem:measure_of_epsilon_neighborhood_of_set_is_huge}
\;\\Let $n\in\mathbb{N}$ be a positive integer, $\varepsilon\in[0,1]$, and $\Dcal\subseteq  Q_n$. If $\mu\left(\Dcal\right)\meg \tailsize{\varepsilon}{n}$ then
\begin{equation}
\mu\left(\Dcal_{\varepsilon}\right)\meg 1-\tailsize{\varepsilon}{n}.
\end{equation}
\end{lemB}

\begin{proof}
    The proof can be deduced by applying Equation 3.5 found in the body of the proof of Proposition 3.1 of \cite{DK2022}. Specifically if one assumes that $\mu\left(\Acal_{\frac{\varepsilon}{2}}\right)\meg\frac{1}{2}$ the result follows if one applies this equation to the set $\left(\Acal_{\frac{\varepsilon}{2}}\right)_{\frac{\varepsilon}{2}}$. On the other hand assuming that $\mu\left(\Acal_{\frac{\varepsilon}{2}}\right)<\frac{1}{2}$ leads to a contradiction by noticing that $\left(Q_n\setminus\Acal_{\frac{\varepsilon}{2}}\right)_{\frac{\varepsilon}{2}}\cap\Acal=\emptyset$.
\end{proof}

The following notation is helpful in describing the roughly equitable block copies. 

\begin{defnB}
Let $n,k\in\mathbb{N}$, $x\in Q_n$, $\Dcal\subset Q_{kn}$ and $\varepsilon_1\in[0,1]$. We define:
\begin{align}
\Vcal_{\pm\varepsilon_1}^n\left(x\right)&\coloneqq\Set*{x'\in Q_n:\norm{x-x'}\in\left((1-\varepsilon)\frac{n}{2},(1+\varepsilon)\frac{n}{2}\right)}\label{eq:varepsilon_central_neighorhood_of_element}\\
\Vcal_{\pm\varepsilon_1}^{k,n}&\coloneqq\Set*{\left(x_0^i,x_1^i\right)_{i\in[k]}\in(Q_n\times Q_n)^k:\forall i\in[k]\,\left(x_0^i\in \Vcal_{\pm\varepsilon_1}^n\left(x_1^i\right)\right)}\\
\Vcal_{\pm\varepsilon_1}^{k,n}(\Dcal)&\coloneqq\Set*{\left(x_0^i,x_1^i\right)_{i\in[k]}\in\Vcal_{\pm\varepsilon_1}^{k,n}:\forall\alpha\in Q_k\left((x_{\alpha_i}^i)_{i\in[k]}\in\Dcal\right)}\label{eq:set_of_varepsilon_central_pairs}
\end{align}
\end{defnB}

We have the following easy fact about the sizes of $\Vcal_{\pm\varepsilon_1}^n\left(x\right)$ and $\Vcal_{\pm\varepsilon_1}^{k,n}$.

\begin{factB}\label{fact:size_of_varepsilon_neighborhood_and_central_neighborhood_of_element}
Let $n\in\mathbb{N}$, $x\in Q_n$ and $\varepsilon_1\in[0,1]$. Then:
\begin{align}
\mu\left(\Vcal_{\pm\varepsilon_1}^n\left(x\right)\right)&=1-\mu\left([n]^{\mik (1-\varepsilon_1)\frac{n}{2}}\right)-\mu\left([n]^{\meg (1+\varepsilon_1)\frac{n}{2}}\right)=1-2\tailsize{\varepsilon_1}{n}\label{eq:size_of_varepsilon_central_neighorhood_of_element}\\
\mu\left(\Vcal_{\pm\varepsilon_1}^{k,n}\right)&=(1-2\tailsize{\varepsilon_1}{n})^k.\label{eq:size_of_set_of_varepsilon_central_pairs}
\end{align}
\end{factB}

This follows directly from the translation invariance of the measure $\mu$: for a fixed $x$, the map $x'\mapsto x-x'$ is a bijection from $ Q_n$ to itself, and it sends $\Vcal_{\pm\varepsilon_1}^n(x)$ to the set of vectors whose Hamming weight lies in $((1-\varepsilon_1)n/2,(1+\varepsilon_1)n/2)$, whose measure is exactly $1-2\tailsize{\varepsilon_1}{n}$.

We will also need a natural probability measure on $\Vcal_{\pm\varepsilon_1}^{k,n}$ which we define by:
\begin{equation}\label{eq:mu*}
\mu^{k,n}_{\varepsilon_1}\left(\Dcal\right)\coloneqq\frac{\abs*{\Vcal_{\pm\varepsilon_1}^{k,n}\left(\Dcal\right)}}{\abs*{\Vcal_{\pm\varepsilon_1}^{k,n}}}.
\end{equation}

\begin{lemB}[Counting the number of roughly balanced pairs $(x_0^1,x_1^1)$ inside a subset]\label{lem:block_εsubspaces_one_dimensional}
\;\\Let $n\in\mathbb{N}$, $\varepsilon_1\in[0,1]$ and $\Dcal\subset Q_n$. Then it holds that
\[
\mu^{1,n}_{\varepsilon_1}\left(\Dcal\right)\meg\left(1-\frac{2\tailsize{\varepsilon_1}{n}}{\mu\left(\Dcal\right)}\right)\left(\mu\left(\Dcal\right)\right)^2.
\]
In particular \textbf{if} $\mu\left(\Dcal\right)>2\tailsize{\varepsilon_1}{n}$ \textbf{then} $\Vcal_{\pm\varepsilon_1}^{1,n}\left(\Dcal\right)\neq\emptyset$.
\end{lemB}
\begin{proof}
For every $x\in\Dcal$, from Equation \eqref{eq:size_of_varepsilon_central_neighorhood_of_element} of Fact \ref{fact:size_of_varepsilon_neighborhood_and_central_neighborhood_of_element} we have
\[
\mu\bigl(\Vcal_{\pm\varepsilon_1}^n\left(x\right)\bigr)=1-2\tailsize{\varepsilon_1}{n}.
\]
The set $\Vcal_{\pm\varepsilon_1}^n(x)$ misses precisely the points $x'$ with $\norm{x-x'}\mik (1-\varepsilon_1)n/2$ or $\norm{x-x'}\meg (1+\varepsilon_1)n/2$; these two tails together have measure $2\tailsize{\varepsilon_1}{n}$.  Consequently,
\[
\mu\bigl(\Vcal_{\pm\varepsilon_1}^n(x)\cap\Dcal\bigr)\meg \mu(\Dcal)-2\tailsize{\varepsilon_1}{n}.
\]

Now we count pairs.  Every element of $\Vcal_{\pm\varepsilon_1}^{1,n}(\Dcal)$ is of the form $(x,x')$ with $x\in\Dcal$ and $x'\in\Dcal\cap\Vcal_{\pm\varepsilon_1}^n(x)$. Hence
\[
\bigl|\Vcal_{\pm\varepsilon_1}^{1,n}(\Dcal)\bigr| \meg \sum_{x\in\Dcal} \bigl|\Dcal\cap\Vcal_{\pm\varepsilon_1}^n(x)\bigr|
\meg |\Dcal|\cdot\min_{x\in\Dcal}\bigl|\Dcal\cap\Vcal_{\pm\varepsilon_1}^n(x)\bigr|.
\]
Dividing by $|\Vcal_{\pm\varepsilon_1}^{1,n}|=(1-2\tailsize{\varepsilon_1}{n})\,2^{2n}$ (by \eqref{eq:size_of_set_of_varepsilon_central_pairs}) gives
\[
\mu^{1,n}_{\varepsilon_1}(\Dcal)
\meg \frac{|\Dcal|\bigl(\mu(\Dcal)-2\tailsize{\varepsilon_1}{n}\bigr)2^{n}}
            {(1-2\tailsize{\varepsilon_1}{n})2^{2n}}
 = \frac{\mu(\Dcal)\bigl(\mu(\Dcal)-2\tailsize{\varepsilon_1}{n}\bigr)}
        {1-2\tailsize{\varepsilon_1}{n}}.
\]
Since $1-2\tailsize{\varepsilon_1}{n}\mik 1$, we obtain
\[
\mu^{1,n}_{\varepsilon_1}(\Dcal) \meg \mu(\Dcal)\bigl(\mu(\Dcal)-2\tailsize{\varepsilon_1}{n}\bigr)
 = \Bigl(1-\frac{2\tailsize{\varepsilon_1}{n}}{\mu(\Dcal)}\Bigr)\bigl(\mu(\Dcal)\bigr)^2.
\]
In particular, if $\mu(\Dcal)>2\tailsize{\varepsilon_1}{n}$ then $\mu^{1,n}_{\varepsilon_1}(\Dcal)>0$, and therefore $\Vcal_{\pm\varepsilon_1}^{1,n}(\Dcal)\neq\emptyset$.
\end{proof}

The next proposition extends the block construction to $k$ dimensions by induction.  The inductive hypothesis is that under suitable density conditions we can find $k$ pairs $(x_0^i,x_1^i)$ with $\norm{x_0^i-x_1^i}\approx n/2$ such that all $2^k$ concatenations belong to $\Dcal$.

\begin{propB}[Finding roughly equitable block copies inside extremely dense subsets]\label{prop:block_k_subspaces}
$\mbox{}$\newline Let $n,k\in\mathbb{N}$ be such that $k\mik n$, $\varepsilon_1\in[\frac{2}{\sqrt{n}},1]$, and $\Dcal\subseteq\{0,1\}^{kn}$ with $\mu\left(\Dcal\right)=\delta$. \textbf{If} we have
\begin{equation}\label{eq:block_delta_bound_k}
\delta > \delta_k\coloneqq \left(2^k\tailsize{\varepsilon_1}{n}\right)^{\frac{1}{2^{k-1}}}\;\;\;\text{and}\;\;\;3k<\varepsilon_1^2 n
\end{equation}
\textbf{then} $\Dcal$ contains a roughly equitable block copy, i.e. $\Vcal_{\pm\varepsilon_1}^{k,n}\left(\Dcal\right)\neq\emptyset$.
\end{propB}
\begin{proof}
We prove by induction on $k$ that under the stated conditions $\mu^{k,n}_{\varepsilon_1}(\Dcal)>0$, which immediately implies $\Vcal_{\pm\varepsilon_1}^{k,n}(\Dcal)\neq\emptyset$.

\textbf{Base case $k=1$:}  
Lemma~\ref{lem:block_εsubspaces_one_dimensional} gives $\mu_{\varepsilon_1}^{1,n}(\Dcal)\meg (1-2\tailsize{\varepsilon_1}{n}/\delta)\delta^2$.  Because $\delta>\delta_1=2\tailsize{\varepsilon_1}{n}$ (by \eqref{eq:block_delta_bound_k} with $k=1$), the factor in parentheses is positive, so $\mu_{\varepsilon_1}^{1,n}(\Dcal)>0$.

\textbf{Inductive step:} Assume $k\meg 2$ and the statement holds for $k-1$.  Write $Q_{kn}=Q_n\times Q_{(k-1)n}$.  For each $x\in Q_{(k-1)n}$ let $\Dcal_x=\{y\in Q_n:(y,x)\in\Dcal\}$ be the section of $\Dcal$ over $x$.  Define, for a pair $(y,y')\in\Vcal_{\pm\varepsilon_1}^{1,n}$,
\[
V_{(y,y')}\coloneqq\bigl\{x\in Q_{(k-1)n}: \{y,y'\}\subset\Dcal_x\bigr\}.
\]
The key observation is the following double‑counting identity:
\[
\sum_{x\in Q_{(k-1)n}} \bigl|\Vcal_{\pm\varepsilon_1}^{1,n}(\Dcal_x)\bigr|
   = \sum_{(y,y')\in\Vcal_{\pm\varepsilon_1}^{1,n}} |V_{(y,y')}|.
\]
Indeed, the left‑hand side counts triples $(x,y,y')$ with $x\in Q_{(k-1)n}$ and $(y,y')\in\Vcal_{\pm\varepsilon_1}^{1,n}(\Dcal_x)$; the right‑hand side counts the same triples grouped by $(y,y')$ first.  Dividing both sides by $|\Vcal_{\pm\varepsilon_1}^{1,n}|\cdot 2^{(k-1)n}$ yields
\begin{equation}\label{eq:aveDxVyy'}
\ave_{x}\mu_{\varepsilon_1}^{1,n}(\Dcal_x)=\ave_{(y,y')}\mu\bigl(V_{(y,y')}\bigr),
\end{equation}
where the averages are taken with respect to the uniform measures on $Q_{(k-1)n}$ and $\Vcal_{\pm\varepsilon_1}^{1,n}$, respectively.

Now we lower‑bound the left‑hand side.  For each $x$, Lemma~\ref{lem:block_εsubspaces_one_dimensional} gives
\[
\mu_{\varepsilon_1}^{1,n}(\Dcal_x)\meg f\bigl(\mu(\Dcal_x)\bigr)
\quad\text{with}\quad f(t)\coloneqq\Bigl(1-\frac{2\tailsize{\varepsilon_1}{n}}{t}\Bigr)t^{2}\;(t>0).
\]
The function $f$ is convex on $(0,\infty)$.  Moreover, Jensen's inequality applied to the convex function $f$ together with $\ave_x\mu(\Dcal_x)=\delta$ gives:
\[
\ave_{x}\mu_{\varepsilon_1}^{1,n}(\Dcal_x)\meg f(\delta)
        =\Bigl(1-\frac{2\tailsize{\varepsilon_1}{n}}{\delta}\Bigr)\delta^{2}.
\]

Combining this with \eqref{eq:aveDxVyy'} we obtain
\begin{equation}\label{eq:lower_bound_on__Ave_Vyy'}
\ave_{(y,y')}\mu\bigl(V_{(y,y')}\bigr)\meg \Bigl(1-\frac{2\tailsize{\varepsilon_1}{n}}{\delta}\Bigr)\delta^{2}.
\end{equation}

If the right‑hand side of \eqref{eq:lower_bound_on__Ave_Vyy'} exceeds $\delta_{k-1}$, then there exists a pair $(y,y')$ for $\mu(V_{(y,y')})>\delta_{k-1}$.  For such a pair we can apply the inductive hypothesis to the set $V_{(y,y')}\subseteq Q_{(k-1)n}$ and conclude that $\Vcal_{\pm\varepsilon_1}^{k,n}(\Dcal)\neq\emptyset$.

Thus it suffices to verify that
\[
\Bigl(1-\frac{2\tailsize{\varepsilon_1}{n}}{\delta_k}\Bigr)\delta_k^{2} \;>\; \delta_{k-1}.
\]
The following auxiliary claim will help.

\begin{claimB}\label{claim:delta_relations}
For $\varepsilon_1\meg 2/\sqrt{n}$ and $k\meg 2$,
\[
\delta_{k-1}=2^{-1/2^{\,k-2}}\,\delta_k^{2},
\qquad
\frac{2\tailsize{\varepsilon_1}{n}}{\delta_k}\mik (\tailsize{\varepsilon_1}{n})^{1/2}.
\]
\end{claimB}
\begin{proof}[Proof of Claim~\ref{claim:delta_relations}]
The first equality follows directly from the definition of $\delta_k$. For the second inequality we need to show $2\tailsize{\varepsilon_1}{n}^{1/2}\mik \delta_k$.  Raising both sides to the power $2^{k-1}$ gives the equivalent condition
\[
2^{2^{k-1}}\cdot\tailsize{\varepsilon_1}{n}^{\,2^{k-2}}\mik 2^{k}\tailsize{\varepsilon_1}{n}
\;\Longleftrightarrow\;
2^{2^{k-1}-k}\mik \tailsize{\varepsilon_1}{n}^{\,1-2^{k-2}} .
\]
For $k=2$ this reads $2^{2-2}=1\mik \tailsize{\varepsilon_1}{n}^{0}=1$, which is true.  For $k\meg 3$ note the $1-2^{k-2}<0$, and that the inequality becomes
\[
\tailsize{\varepsilon_1}{n}\mik \left(\frac{1}{2}\right)^{\frac{2^{k-1}-k}{2^{k-2}-1}}.
\]
The right‑hand side is at least $\frac{1}{4}$ for $k\meg 3$, furthermore by Hoeffding's bound (Lemma~\ref{lem:Hoeffding_bounds_for_tails}), $\tailsize{\varepsilon_1}{n}\mik\exp(-\varepsilon_1^{2}n/2)$.  Now the condition $\varepsilon_1\meg \frac{2}{\sqrt{n}}$ yields the desired bound.
\end{proof}

Using Claim~\ref{claim:delta_relations},
\[
\Bigl(1-\frac{2\tailsize{\varepsilon_1}{n}}{\delta_k}\Bigr)\delta_k^{2}
\meg \bigl(1-(\tailsize{\varepsilon_1}{n})^{1/2}\bigr)\delta_k^{2}
    = \bigl(1-(\tailsize{\varepsilon_1}{n})^{1/2}\bigr)2^{1/2^{k-2}}\delta_{k-1}.
\]
Now again by Hoeffding's bound and the constraint $3k<\varepsilon_1^2 n$, 
we have that $\tailsize{\varepsilon_1}{n}^{\frac{1}{2}}\mik e^{-\frac{3k}{4}}$, hence it is enough to show that:
\begin{equation}
\bigl(1-e^{-\frac{3k}{4}}\bigr)2^{1/2^{k-2}}\geq 1.
\end{equation}
This holds for every $k\geq 2$ and consequently the product exceeds $\delta_{k-1}$, completing the induction.
\end{proof}

The above block construction can be used to bound $\Emb(1+\varepsilon,\delta,k)$ but gives us very bad bounds which are of the order of $2^k$. The next proposition shows how to exploit the concentration of measure phenomenon and the nice block structure that is guaranteed by Proposition \ref{prop:block_k_subspaces} to obtain much better bounds.

\begin{propB}[Small block error and small perturbation error yield small distortion]\label{prop:subspace_in_Dε_then_subspace_in_D_two_eps}
\;\\Let \(n,k\in\mathbb{N}\) and \(\mathcal{D}\subset\{0,1\}^{kn}\). Let \(\varepsilon_1,\varepsilon_2\in(0,1)\) and set
\(
\varepsilon \coloneqq 9\left(\varepsilon_1+\varepsilon_2 k\right).
\)
Assume that
\(
\varepsilon \mik \frac{1}{4}.
\)
\\\textbf{If} we have a roughly equitable block copy inside the $\varepsilon_2$-neighborhood of $\Dcal$, i.e. $\mathcal{V}_{\pm\varepsilon_1}^{k,n}\bigl(\mathcal{D}_{\varepsilon_2}\bigr)\neq\emptyset$, 
\\\textbf{then} 
$\Dcal$ contains a $(1+\varepsilon)$-distorted and $r:=\frac{n}{2}\bigl(1-\varepsilon_1-4\varepsilon_2 k\bigr)$-rescaled copy of the Hamming cube $Q_k$, i.e. 
there exists an injective function \(f: Q_k\rightarrow\mathcal{D}\) such that for every \(\alpha,\alpha'\in Q_k\) it holds that
\begin{equation}\label{eq:subspace_in_Dε_then_embedding_in_D_two_eps}
r \,\norm*{\alpha-\alpha'}
\;\mik\;
\norm*{f\left(\alpha\right)-f\left(\alpha'\right)}
\;\mik\;
r\left(1+\varepsilon\right)\norm*{\alpha-\alpha'}.
\end{equation}
\end{propB}

\begin{proof}
Since \(\mathcal{V}_{\pm\varepsilon_1}^{k,n}\bigl(\mathcal{D}_{\varepsilon_2}\bigr)\neq\emptyset\), there exist pairs
\[
(x_0^i,x_1^i)\in Q_n\times Q_n\quad\text{for each }i\in[k]
\]
such that
\begin{equation}\label{eq:pair_distance_bounds_two_eps_new}
\frac{n}{2}(1-\varepsilon_1) \;\mik\; \norm{x_0^i-x_1^i} \;\mik\; \frac{n}{2}(1+\varepsilon_1)
\quad\forall i\in[k],
\end{equation}
and for every \(\alpha=(\alpha_i)_{i\in[k]}\in Q_k\), the concatenated vector
\[
X(\alpha)\coloneqq (x_{\alpha_i}^i)_{i\in[k]}\in\mathcal{D}_{\varepsilon_2}.
\]
By the definition of \(\mathcal{D}_{\varepsilon_2}\), for each \(\alpha\in Q_k\) there exists \(Y(\alpha)\in\mathcal{D}\) such that
\begin{equation}\label{eq:approximation_error_two_eps_new}
\norm{Y(\alpha)-X(\alpha)} \;\mik\; \varepsilon_2 k n.
\end{equation}
Fix one such \(Y(\alpha)\) for each \(\alpha\) and define \(f(\alpha)\coloneqq Y(\alpha)\).

\smallskip
\noindent\textbf{Step 1: Adjacent vertices.}
Suppose \(\alpha,\alpha'\in Q_k\) differ in exactly one coordinate. Then $X(\alpha)$ and $X(\alpha')$ differ in exactly one block, and by \eqref{eq:pair_distance_bounds_two_eps_new},
\[
\norm{X(\alpha)-X(\alpha')} = \norm{x_0^i-x_1^i}
\in\left[\frac{n}{2}(1-\varepsilon_1),\,\frac{n}{2}(1+\varepsilon_1)\right].
\]

Using the triangle inequality and \eqref{eq:approximation_error_two_eps_new}, we obtain
\begin{align*}
\norm{f(\alpha)-f(\alpha')}
&\mik \norm{f(\alpha)-X(\alpha)} + \norm{X(\alpha)-X(\alpha')} + \norm{X(\alpha')-f(\alpha')} \\
&\mik \varepsilon_2kn + \frac{n}{2}(1+\varepsilon_1) + \varepsilon_2kn 
= \frac{n}{2}\bigl(1+\varepsilon_1+4\varepsilon_2k\bigr),
\end{align*}
and similarly
\begin{align*}
\norm{f(\alpha)-f(\alpha')}
&\meg \norm{X(\alpha)-X(\alpha')} - \norm{f(\alpha)-X(\alpha)} - \norm{f(\alpha')-X(\alpha')} \\
&\meg \frac{n}{2}(1-\varepsilon_1) - 2\varepsilon_2kn 
 = \frac{n}{2}\bigl(1-\varepsilon_1-4\varepsilon_2k\bigr) 
 = r.
\end{align*}
Thus, whenever $\norm{\alpha-\alpha'}=1$,
\begin{equation}\label{eq:distance_one_bounds_two_eps_new}
r \;\mik\; \norm{f(\alpha)-f(\alpha')}
\;\mik\; \frac{n}{2}\bigl(1+\varepsilon_1+4\varepsilon_2k\bigr).
\end{equation}

\smallskip
\noindent\textbf{Step 2: General pairs.}
Let now \(\alpha,\alpha'\in Q_k\) be arbitrary and set $d\coloneqq\norm{\alpha-\alpha'}$.  Choose a geodesic path
\(
\alpha=\alpha^0,\alpha^1,\dots,\alpha^d=\alpha'
\)
in the hypercube with consecutive vertices differing in one coordinate.  Summing the block‑wise distances along this path gives
\begin{equation}\label{eq:X_distance_bounds_two_eps_new}
\frac{n}{2}(1-\varepsilon_1)d \;\mik\; \norm{X(\alpha)-X(\alpha')}
\;\mik\; \frac{n}{2}(1+\varepsilon_1)d.
\end{equation}
Applying the triangle inequality together with \eqref{eq:approximation_error_two_eps_new} yields
\begin{align*}
\norm{f(\alpha)-f(\alpha')}
&\mik \norm{f(\alpha)-X(\alpha)} + \norm{X(\alpha)-X(\alpha')} + \norm{X(\alpha')-f(\alpha')} \\
&\mik 2\varepsilon_2kn + \frac{n}{2}(1+\varepsilon_1)d 
\mik \frac{n}{2}\bigl(1+\varepsilon_1+4\varepsilon_2k\bigr)d,
\end{align*}
because $2\varepsilon_2kn = (n/2)\cdot 4\varepsilon_2k \mik (n/2)\cdot 4\varepsilon_2k\,d$ (since $d\ge1$).  For the lower bound,
\begin{align*}
\norm{f(\alpha)-f(\alpha')}
&\meg \norm{X(\alpha)-X(\alpha')} - \norm{f(\alpha)-X(\alpha)} - \norm{f(\alpha')-X(\alpha')} \\
&\meg \frac{n}{2}(1-\varepsilon_1)d - 2\varepsilon_2kn 
= \frac{n}{2}\bigl((1-\varepsilon_1)d - 4\varepsilon_2k\bigr) 
\meg \frac{n}{2}\bigl(1-\varepsilon_1-4\varepsilon_2k\bigr)d 
= r\,d,
\end{align*}
where the last inequality uses $d\ge1$ to replace $(1-\varepsilon_1)d-4\varepsilon_2k$ by $(1-\varepsilon_1-4\varepsilon_2k)d$.

Thus for all \(\alpha,\alpha'\in Q_k\),
\begin{equation}\label{eq:raw_bounds_two_eps_new}
r\,\norm{\alpha-\alpha'}
\;\mik\;
\norm{f(\alpha)-f(\alpha')}
\;\mik\;
\frac{n}{2}\bigl(1+\varepsilon_1+4\varepsilon_2k\bigr)\norm{\alpha-\alpha'}.
\end{equation}

\smallskip
\noindent\textbf{Step 3: Bounding the upper factor by \(r(1+\varepsilon)\).}
Recall that $r = \frac{n}{2}(1-t)$ with $t\coloneqq \varepsilon_1+4\varepsilon_2k$, and the upper factor in \eqref{eq:raw_bounds_two_eps_new} is $\frac{n}{2}(1+t)$.  Set $u\coloneqq \varepsilon_1+\varepsilon_2k$; then by definition $\varepsilon=9u$.  The assumption $\varepsilon\le1/4$ gives $u\le 1/36$.  Moreover,
\[
t = \varepsilon_1+4\varepsilon_2k \mik \varepsilon_1+4(\varepsilon_1+\varepsilon_2k) = 5\varepsilon_1+4\varepsilon_2k \mik 4u,
\]
where the last inequality uses $\varepsilon_1\le u$.

We need to show
\[
\frac{1+t}{1-t} \mik 1+9u = 1+\varepsilon .
\]
Since the right‑hand side is a decreasing function of $t$ for fixed $u$, it suffices to check the inequality at the largest possible $t$, namely $t=4u$.  Thus we verify that for $0\le u\le 1/36$,
\[
1+4u \mik (1-4u)(1+9u).
\]
Expanding the right side gives $1+5u-36u^{2}$, so the inequality reduces to
\[
1+4u \mik 1+5u-36u^{2}
\;\Longleftrightarrow\;
0 \mik u-36u^{2}
\;\Longleftrightarrow\;
u(1-36u)\ge0,
\]
which holds precisely when $u\le 1/36$.  Therefore
\[
\frac{n}{2}(1+\varepsilon_1+4\varepsilon_2k) = \frac{n}{2}(1+t) \mik \frac{n}{2}(1-t)(1+9u) = r(1+\varepsilon).
\]

Combining this with \eqref{eq:raw_bounds_two_eps_new} yields the desired bi‑Lipschitz bounds \eqref{eq:subspace_in_Dε_then_embedding_in_D_two_eps}.  Finally, $r>0$ because $t\le4u\le4/36=1/9$, so $1-t\ge8/9>0$; consequently $f$ is injective (distinct points have positive distance).
\end{proof}

Now we put the two propositions together to obtain the main theorem of this section.

\begin{thmB}\label{thm:epsilon_distortion}
Let $k\in\mathbb{N}$, $\delta\in(0,1)$ and $\varepsilon\in(0,\frac{1}{4})$. Define
\begin{equation}\label{eq:thm_epsilon_distortion}
N(\varepsilon,\delta,k)=C\varepsilon^{-2}k^3\log\left(\frac{1}{\delta}\right),
\end{equation}
with $C=4\cdot 18^2$. Then for every $N\in\mathbb{N}$ with $N\meg N(\varepsilon,\delta,k)$ and every $\Dcal\subseteq Q_N$ with $\mu\left(\Dcal\right)>\delta$ there exists a $(1+\varepsilon)$-bi-Lipschitz embedding $f: Q_k\to\Dcal$ with rescaling $r = (1-\varepsilon_1 - 4 \varepsilon_2 k) N/(2k)$.
Consequently,
\[
\Emb(1+\varepsilon,\delta,k)\mik C\varepsilon^{-2}k^{3}\log\!\Bigl(\frac{1}{\delta}\Bigr).
\]
\end{thmB}

The rescaling estimate is crucial in the proof of Theorem \ref{thm:applicationToBMWType}.

\begin{proof}
Let $N\in\mathbb{N}$ with $N\meg N(\varepsilon,\delta,k)$ and $\Dcal\subseteq Q_N$ with $\mu\left(\Dcal\right)>\delta$, first note that we may assume that $\delta<\frac{1}{4}$ else we can work on a subset of $\mathcal{D}$. Furthermore note we can assume that $N=nk$ for $n=\left\lfloor\frac{N}{k}\right\rfloor\in\mathbb{N}$ with $n\meg C\varepsilon^{-2}k^2\log\left(\frac{1}{\delta}\right)-1$, this can be done with a standard averaging argument. As a first remark note that the constraint \ref{eq:thm_epsilon_distortion} we have imposed on $N$ implies that:
\begin{equation}\label{eq:thm_epsilon_distortion_epsilon_constrain}
\varepsilon\meg\sqrt{\frac{Ck^3\log(\frac{1}{\delta})}{N}}.
\end{equation}
The goal is to apply Propositions \ref{prop:block_k_subspaces} and \ref{prop:subspace_in_Dε_then_subspace_in_D_two_eps}. To this end we define $\varepsilon_1=\frac{\varepsilon}{18}$ and $\varepsilon_2=\frac{\varepsilon}{18k}$ and note that the following equation is satisfied 
\begin{equation}\label{eq:thm_epsilon_distortion_constrain1_fouskoma}
9(\varepsilon_1+\varepsilon_2 k)= \varepsilon<\frac{1}{4}.
\end{equation}
So we may apply Proposition \ref{prop:subspace_in_Dε_then_subspace_in_D_two_eps} as long as the assumptions for Proposition \ref{prop:block_k_subspaces} are satisfied. First note that $n\geq k$. Furthermore from the assumption that $\delta<\frac{1}{4}$
\begin{equation}\label{eq:thm_epsilon_distortion_constrain1_block}
\varepsilon_1^2 n=\frac{\varepsilon^2}{18^2}\left\lfloor\frac{N}{k}\right\rfloor\meg \frac{\varepsilon^2}{18^2} \left(C\varepsilon^{-2}k^2\log\left(\frac{1}{\delta}\right)-1\right)>7k^2>3k.
\end{equation}
In particular we also have that: 
\begin{equation}\label{eq:thm_epsilon_distortion_constrain2_block}
\varepsilon_1>k\sqrt{\frac{7}{n}}>\frac{2}{\sqrt{n}}.
\end{equation}
So we only need to verify that: 
\begin{equation}\label{eq:thm_epsilon_distortion_constrain4_block_main}
\mu\left(\Dcal_{\varepsilon_2}\right)>\left(2^k\tailsize{\varepsilon_1}{n}\right)^{2^{-(k-1)}}.
\end{equation}

From \eqref{eq:thm_epsilon_distortion}, the definition of $\varepsilon_2$, the definition of $\tailsize{\varepsilon_2}{N}$ and Hoeffding’s bound in Equation \ref{eq:Hoeffding_bounds_for_uniform_tails} we have:
\begin{align}\label{eq:thm_epsilon_distortion_from_N_to_delta}
C\varepsilon^{-2}k^3\log\left(\frac{1}{\delta}\right)		&\mik N \Longleftrightarrow 
\log\left(\frac{1}{\delta}\right)
\mik \frac{\varepsilon^2 N}{Ck^3}\\
&\Longleftrightarrow\delta	\meg \exp\left(-\frac{\varepsilon^2}{18^2k^2}\frac{N}{2}\right)^{\frac{2\cdot 18^2}{Ck}} \meg \tailsize{\varepsilon_2}{N}^{\frac{2\cdot 18^2}{Ck}}.
\end{align}
From the definition of $C$ we have that $\frac{2\cdot 18^2}{Ck}\mik 1$ hence by Lemma \ref{lem:measure_of_epsilon_neighborhood_of_set_is_huge} and another application of Hoeffding's bound we have that:
\begin{equation}\label{eq:thm_epsilon2_distortion_fouskoma_size}
\mu\left(\Dcal_{\varepsilon_2}\right)\meg 1-\tailsize{\varepsilon_2}{N}\meg 1-\exp\left(-\varepsilon_2^2\frac{N}{2}\right)=1-\exp\left(-\varepsilon^2\frac{N}{2\cdot 18^2k^2}\right).
\end{equation}
Plugging in the constraint  \ref{eq:thm_epsilon_distortion} for $N$ and noting that for $\delta<\frac{1}{4}$, it holds that $\log\left(\frac{1}{\delta}\right)>2$ we get that:
\begin{equation}\label{eq:thm_epsilon_distortion_constrain4_block_main_new}
\mu\left(\Dcal_{\varepsilon_2}\right)\meg 1-\exp\left(-\varepsilon^2\frac{C\varepsilon^{-2}k^3\log\left(\frac{1}{\delta}\right)}{2\cdot 18^2k^2}\right)>1-\exp\left(-4k\right).
\end{equation}
So for Equation \ref{eq:thm_epsilon_distortion_constrain4_block_main} to hold it is enough to show that: 
\begin{equation}\label{eq:thm_epsilon_distortion_constrain4_block_main_new_new}
1-\exp\left(-4k\right)\meg \left(2^k\tailsize{\varepsilon_1}{n}\right)^{2^{-(k-1)}}
\end{equation}
Taking logarithms and using one more time Hoeffding's bound for $\tailsize{\varepsilon_1}{n}$, we have that it is enough to show that:
\begin{equation}\label{eq:thm_epsilon_distortion_constrain4_block_main_new_new_new}
\log\left(1-\exp\left(-4k\right)\right)\meg \frac{1}{2^{k-1}}\left(k-\varepsilon_1^2\frac{n}{2\ln2}\right),
\end{equation}
we rearrange and have that the conclusion of the theorem will follow if we prove that:
\begin{equation}\label{eq:thm_epsilon_distortion_constrain4_block_main_new_new_new_new}
2^{k-1}\log\left(\frac{1}{1-\exp\left(-4k\right)}\right)\mik \left(\varepsilon_1^2\frac{n}{2\ln2}-k\right)
\end{equation}
Now note that from the calculations we made in Equation \ref{eq:thm_epsilon_distortion_constrain1_block} $2k^2< \varepsilon_1^2\frac{n}{2\ln2}-k$. Furthermore with an application of the inequality $\log(1+x)\mik\frac{x}{\ln 2}$ for $x\meg -1$ we have that
\[\log\left(\frac{1}{1-\exp\left(-4k\right)}\right)=\log\left(1+\frac{\exp\left(-4k\right)}{1-\exp\left(-4k\right)}\right)\mik\frac{\exp\left(-4k\right)}{\ln 2\left(1-\exp\left(-4k\right)\right)}.\] 
From the above consideration Equation \ref{eq:thm_epsilon_distortion_constrain4_block_main_new_new_new_new} will follow if: 
\begin{equation}\label{eq:thm_epsilon_distortion_constrain4_block_main_new_new_new_new_new}
2^{k-1}\frac{\exp\left(-4k\right)}{\ln 2\left(1-\exp\left(-4k\right)\right)}\mik 2k^2.
\end{equation}
This holds and the proof is complete.
\end{proof}

\section{\textbf{Lower bounds for undistorted cubes}}\label{sec:lowerBounds}

\subsubsection{\textbf{Lower bounds (finding dense subsets with no cubes)}}
We use the probabilistic method. We flip $2^M$ i.i.d. Bernoulli($\delta$) random variables, one for each $x \in Q_N$, and set $\Dcal$ to be the set of all points $x$ whose Bernoulli random variable is $1$.
We get a random subset with expected density $\e\mu(\Dcal) = \delta$.
Our goal is to show that $\Dcal$ does not contain any undistorted or bounded rescaling undistorted copy of $Q_k$.
\\We can bound the probability that $\Dcal$ contains a copy because
\underline{undistorted copies are rigid}.
We show that every undistorted map $f: Q_k \to Q_N$  must be of the form:
$$f(\alpha) = b + \sum_{i\in[k]:\alpha_i=1}\mathbf{1}_{I_i}\;\;\;\text{for all }\alpha \in Q_k$$
where $b\in Q_N$ and $I_1,\ldots,I_k$ are disjoint subsets of $[N]$ of equal size. The rescaling factor of the embedding is $r = |I_1| = \ldots = |I_k|$.

For the lower bound for undistorted copies, we simply apply a union bound over all copies.
For the lower bound for undistorted copies of bounded rescaling, we observe that for each fixed copy, the number of other copies it intersects is small and apply the Lovász local lemma. We start with the characterization of the undistorted copies.

\begin{propB}[Characterization of undistorted copies]\label{prop:classification_zero_distorted_copies}
Let $k,N\in\mathbb{N}$ with $N\meg k$ and $1\leq r \leq N$ and let $f: Q_k\to Q_N$ for which there exists $r\in\mathbb{N}$ such that for every $\alpha,\alpha'\in Q_k$ it holds:
\begin{equation}\label{eq:classification_zero_distorted_copies_embedding}
\norm{f(\alpha)-f(\alpha')}=r\norm{\alpha-\alpha'}.
\end{equation}
Then there exists $b\in Q_N$ and $k$ disjoint subsets $I_1,\ldots,I_k$ of $[N]$, each of size $r$, such that the map $f$ is given by:
\begin{equation}\label{eq:classification_zero_distorted_copies_embedding_characterization_of_embedding}
f(\alpha)=b+\!\!\!\!\!\!\!\sum_{i\in[k]:\alpha_i=1}\mathbf{1}_{I_i}.
\end{equation}      
\end{propB}
\begin{proof}
Let $f: Q_k\to Q_N$ be as in the hypothesis of Proposition \ref{prop:classification_zero_distorted_copies}. Let $0^k$ be the zero vector in $ Q_k$. Without loss of generality we may assume that $f(0^k)=0^N$ (otherwise translate by $b=f(0^k)$). For every $i\in[k]$ let $e_i\in Q_k$ be the standard basic vector whose $i$-th coordinate is $1$ and all others are $0$. For each $\alpha\in Q_k$, denote by
\[
A_\alpha=\{m\in[N]: f(\alpha)_m=1\}
\]
the support of $f(\alpha)$. Define $I_i\subset[N]$ by
\[
I_i=A_{e_i}=\Set*{m\in[N]:f(e_i)_m=1}.
\]
Note that for every $i\in[k]$ we have that $\norm{f(e_i)}=\abs{A_{e_i}}=r$ and for every $i\neq j$ with $i,j\in[k]$ it holds that $\norm{f(e_i)-f(e_j)}=2r$. This implies that $I_i\cap I_j=\emptyset$ for $i\neq j$. Furthermore, $\norm{f(e_i+e_j)}=r\norm{e_i+e_j}=2r$.

We compute
\[
r=r\norm{e_i+e_j-e_i}=\norm{f(e_i+e_j)-f(e_i)}.
\]
On the other hand
\[
\norm{f(e_i+e_j)-f(e_i)}=
\abs{A_{e_i+e_j}\triangle A_{e_i}}
=\abs{A_{e_i+e_j}}+\abs{A_{e_i}}-2\abs{A_{e_i+e_j}\cap A_{e_i}}.
\]
Since $\abs{A_{e_i+e_j}}=2r$ and $\abs{A_{e_i}}=r$, we deduce that
\[
r=2r+r-2\abs{A_{e_i+e_j}\cap A_{e_i}}
\quad\Longrightarrow\quad
\abs{A_{e_i+e_j}\cap A_{e_i}}=r.
\]
Similarly $\abs{A_{e_i+e_j}\cap A_{e_j}}=r$, so we can compute that
\[
\norm{f(e_i+e_j)-(f(e_i)+f(e_j))}
=\abs{A_{e_i+e_j}\triangle(A_{e_i}\cup A_{e_j})}=0.
\]
Thus $f(e_i+e_j)=f(e_i)+f(e_j)$ for all $i\neq j$. Because $i$ and $j$ were arbitrary and the supports $I_i=A_{e_i}$ are pairwise disjoint, we can now prove by induction on the Hamming weight of $\alpha$ that
\[
f(\alpha)=\sum_{i:\alpha_i=1}f(e_i)
=\sum_{i:\alpha_i=1}\mathbf{1}_{I_i}
\]
for all $\alpha\in Q_k$, showing that $f$ is affine-linear with disjointly supported basis vectors, as claimed.
\end{proof}

\begin{remB}\label{rem:zero_distorted_copies_are_affine_subspaces}
Proposition \ref{prop:classification_zero_distorted_copies} shows that every isometric copy of $Q_k$ inside $Q_N$ (with a fixed scaling $r$) is, up to translation, an affine subspace spanned by $k$ disjointly supported vectors of weight $r$.  In particular, the number of such copies inside $Q_N$ can be crudely bounded as follows: choose a translation vector $b$ ($2^N$ possibilities) and for each coordinate $m\in[N]$ decide whether it is fixed by $b$ or belongs to one of the $k$ disjoint blocks $I_i$; this gives at most $(k+1)^N$ choices.  Hence there are no more than $(2(k+1))^N$ rescaled isometric copies of $Q_k$ in $Q_N$.
By the same reasoning, there are no more than $2^N R N^{Rk}$ rescaled isometric copies with rescaling $\leq R$ (there are $2^N$ choices for $b$ and for each of the $R$ possible rescalings there are $\leq N^{Rk}$ possible sets $I_1,...,I_k$.)
\end{remB}
$\mbox{}$\\
We now use a probabilistic construction to show that $\Emb(1,\delta,k)$ must be exponentially large in $k$.

\begin{thmB}[Lower bound for undistorted copies]\label{thm:lower_bound_on_Emb(1,delta,k)_weaker}
Let $k\ge 1$ and $0 < \delta < 1/2$. Then 
\begin{equation}\label{eq:thm53-weaker-claim}
\Emb(1,\delta,k) \;\meg\; \frac{1}{100}\,\frac{2^k \log\bigl(\tfrac{1}{\delta}\bigr)}{\log (k+1)}.
\end{equation}
\end{thmB}

\begin{proof}
Let $N$ be an integer such that
\begin{equation}\label{eq:thm53-weaker-N}
N \;\mik\; \frac{1}{100}\,\frac{2^k \log\bigl(\tfrac{1}{\delta}\bigr)}{\log (k+1)}.
\end{equation}
We show that there exists $\Dcal\subseteq Q_N$ with $\mu(\Dcal)>\delta$ and containing no rescaled isometric copy of $Q_k$.

If
\begin{equation}\label{eq:thm53-weaker-smallN}
N<k-1+\log\bigl(\tfrac{1}{\delta}\bigr),
\end{equation}
then an affine $(k-1)$-dimensional subcube $\Dcal\subseteq Q_N$ satisfies $\mu(\Dcal)=2^{k-1-N}>\delta$ and $|\Dcal|=2^{k-1}<2^k$, so it contains no rescaled isometric copy of $Q_k$. Thus we may assume
\begin{equation}\label{eq:thm53-weaker-largeN}
N\;\meg\;k-1+\log\bigl(\tfrac{1}{\delta}\bigr).
\end{equation}
For $1\le k\le 8$ one has \(\frac{1}{100}\frac{2^k}{\log(k+1)}<1\), so \eqref{eq:thm53-weaker-N} and \eqref{eq:thm53-weaker-largeN} cannot hold simultaneously. Hence we may also assume $k\ge 9$.

Now choose a random subset $\Dcal\subseteq Q_N$ by including each point independently with probability $p\coloneqq 3\delta/2$. Then $\mathbb{E}|\Dcal|=p2^N$, and by the lower-tail Chernoff bound for sums of independent Bernoulli random variables \cite[Theorem~4.5]{MitzenmacherUpfal2005},
\begin{equation}\label{eq:thm53-weaker-chernoff}
\prob\bigl(|\Dcal|\mik \delta 2^N\bigr)
\mik
\exp\!\Bigl(-\frac{\delta 2^N}{12}\Bigr)
\mik
\exp\!\Bigl(-\frac{2^{k-1}}{12}\Bigr)
<\frac14,
\end{equation}
where the second inequality follows from \eqref{eq:thm53-weaker-largeN}.

Let $\mathcal M$ be the collection of all rescaled isometric copies of $Q_k$ inside $Q_N$. By Remark~\ref{rem:zero_distorted_copies_are_affine_subspaces}, $|\mathcal M|\mik (2(k+1))^N$. Let
\[
P_{\text{bad}}\coloneqq \prob\bigl(\exists M\in \mathcal M:\ M\subseteq \Dcal\bigr).
\]
We claim that $P_{\text{bad}}<1/4$. For each fixed $M\in\mathcal M$ one has
\[
\prob(M\subseteq \Dcal)=\bigl(\tfrac{3\delta}{2}\bigr)^{2^k},
\]
so by the union bound,
\begin{equation}\label{eq:thm53-weaker-secondterm}
P_{\text{bad}}\mik (2(k+1))^N\bigl(\tfrac{3\delta}{2}\bigr)^{2^k}.
\end{equation}
Since $e^{-2}<1/4$, it is enough to show that
\[
(2(k+1))^N\bigl(\tfrac{3\delta}{2}\bigr)^{2^k}\mik e^{-2}.
\]
Taking logarithms, this is equivalent to
\[
N\log\bigl(2(k+1)\bigr)+2^k\log\bigl(\tfrac{3\delta}{2}\bigr)\mik -2.
\]
This indeed holds by \eqref{eq:thm53-weaker-N}, the elementary inequalities
\[
\log\bigl(\tfrac{3\delta}{2}\bigr)=\log\bigl(\tfrac32\bigr)-\log\bigl(\tfrac1\delta\bigr)\mik -\frac25\log\bigl(\tfrac1\delta\bigr),
\qquad
\log\bigl(2(k+1)\bigr)\mik 2\log(k+1),
\]
and the facts that $k\ge 9$ and $\log(1/\delta)>1$, since these give
\[
N\log\bigl(2(k+1)\bigr)+2^k\log\bigl(\tfrac{3\delta}{2}\bigr)
\mik
\frac{1}{50}\,2^k\log\bigl(\tfrac1\delta\bigr)-\frac25\,2^k\log\bigl(\tfrac1\delta\bigr)
=
-\frac{19}{50}\,2^k\log\bigl(\tfrac1\delta\bigr)
\mik -2.
\]
Hence $P_{\text{bad}}<1/4$. Combining this with \eqref{eq:thm53-weaker-chernoff}, we obtain
\[
\prob\bigl(|\Dcal|\mik \delta 2^N\bigr)+\prob\bigl(\exists M\in \mathcal M:\ M\subseteq \Dcal\bigr)
<\frac14+\frac14<1.
\]
Thus there exists $\Dcal\subseteq Q_N$ with $\mu(\Dcal)>\delta$ and containing no rescaled isometric copy of $Q_k$. Therefore $\Emb(1,\delta,k)>N$. Since $N$ was arbitrary subject to \eqref{eq:thm53-weaker-N}, the proof is complete.
\end{proof}

\begin{thmB}[Lower bound for undistorted copies of bounded rescaling]\label{thm:lower_bound_on_boudned_rescaling}
Let $k\ge 4$ and $0 < \delta < 1/2$ and $R \geq 1$. Then 
\[
\Emb^{\leq R}(1,\delta,k) \;\meg\; \left(\dfrac{2^{-k}}{12R}\right)^{1/Rk} 2^{\log(1/\delta)\; 2^k/(Rk)}.
\]
\end{thmB}
We will use the following standard quantitative version of the Lovász-Local-Lemma:

\begin{lemB}[detailed Lovász-Local-Lemma]\label{lem:detailedLLL}
    \;\\Let $A_1,...,A_J$ be a sequence of events such that each event occurs with probability at most p and such that each event is independent of all the other events except for at most d of them.
    \\\textbf{If} the condition $e\cdot p\cdot d\leq 1$ holds, 
    \\\textbf{then}  there is a nonzero probability that none of the events occurs, in fact:
    $$\mathbb{P}[\cap_{j=1}^J A_j^c] \geq (1-p)^J.$$
\end{lemB}
\begin{proof}
    See Chapter 5 Lemma 5.1.1 in the book \cite{alon2016probabilistic}.
\end{proof}

\begin{proof}[Proof of Theorem \ref{thm:lower_bound_on_boudned_rescaling} (Lower bound for undistorted copies of bounded rescaling)]
    Let $N$ be an integer to be chosen later.
    Exactly as in Theorem \ref{thm:lower_bound_on_Emb(1,delta,k)_weaker} consider a random subset $\Dcal\subseteq Q_N$ obtained by including each point independently with probability $p\coloneqq 3\delta/2$.
    That way $\mathbb{E}|\Dcal|=p\,2^N$, and by the multiplicative Chernoff bound,
\[
\prob\bigl(|\Dcal|\mik \delta 2^N\bigr)
   \leq \exp\!\Bigl(-\frac{\delta 2^N}{12}\Bigr).
\]

Let $M_{\leq R}$ be the collection of all isometric copies of $Q_k$ inside $Q_N$ with rescaling $\leq R$.  
By Remark~\ref{rem:zero_distorted_copies_are_affine_subspaces}, $|M_{\leq R}|\mik 2^N R N^{Rk}$.
For a copy $M\in M_{\leq R}$, the probability that $M\subseteq\Dcal$ equals $p^{2^k}$.

Now fix any copy $M_1 \in M_{\leq R}$ and count the number of copies $M_2$ such that $M_1 \cap M_2 \neq \emptyset$.
\\We fix $x\in M_1$ and count the number of copies $M_2$ with
$x \in M_2$. This number is
$$|\{M_2 \in M_{\leq R}:x \in M_2\}|\leq \sum_{r=1}^R \binom{N}{r}^k \leq R N^{Rk}$$
Going over all $x \in M_1$ we get: $d := |\{M_2 \in M_{\leq R}:M_1 \cap M_2 \neq \emptyset \}| \leq 2^k R N^{Rk}$.

Finally, we check the numerics and apply the Lov\`{a}sz Local Lemma.
The assumption of Lemma \ref{lem:detailedLLL} is $edp^{2^k} \leq 1$ and becomes
$$e\;2^k R N^{Rk} \left(\dfrac{3\delta}{2}\right)^{2^k} \leq 1 \iff N^{Rk} \leq \dfrac{2^{-k}}{4R} e^{2^k \ln(2/3\delta)}$$
$$\iff N \leq \dfrac{2^{-k}}{4R} e^{\ln(2/3\delta)\; 2^k/(Rk)} = \exp{({\ln(1/\delta) \exp{(\Theta(k))}})}.$$
We pick $N$ such that this constraint is satisfied.
By the Lov\`{a}sz Local Lemma we arrive that
$$\mathbb{P}[\Dcal\;\text{ contains no copy}] > (1-(3\delta/2)^{2^{k}})^{2^N R N^{Rk}}\geq e^{-2 (3\delta/2)^{2^k} 2^N R N^{Rk}}$$
(where the second inequality follows from the fact that for all $0<x<1/2$, $e^{-2x}\leq 1-x$).

In order to conclude from a union bound that
$$\mathbb{P}[\mu(\Dcal)>\delta\;\text{ and }\Dcal\;\text{ contains no copy}] > 0$$
we need to pick $N$ such that the following additional constraint is satisfied:
$$e^{-\delta 2^N /12} <^{?} e^{-2 \delta^{2^k} 2^N R N^{Rk}} \iff \delta /12 > 2 \delta^{2^k} R N^{Rk} \iff N \leq \left(\dfrac{2^{-k}}{12R}\right)^{1/Rk} e^{\ln(2/3\delta)\; (2^k - 1)/(Rk)}.$$

\end{proof}

\section{\textbf{Metric embeddings of path spaces into dense subsets of path spaces}}\label{sec:pathSpaces}

In this section we prove Theorem \ref{thm:pathSpaces}. We restate the theorem below with explicit constants.

\begin{thmB}[Metric embeddings of paths into dense subsets of paths]
\;\\For all $k\geq 3, 0 < \delta < 1, 0 < \varepsilon \leq 1$, we have:
    $$e^{0.1 k \log(1/\delta)} \leq Path(1+\varepsilon,\delta,k) \leq e^{100 k \varepsilon^{-1} \log(1/\delta)}.$$
\end{thmB}

The proof of the upper bound is similar to the proof in \cite{dumitrescu2010approximate}.
We use a porosity argument.
We show that for any interval $I$, if $\Dcal \cap I$ has no $(1+\varepsilon)$-distorted copy of $[k]$ then $\Dcal \cap I$ is disjoint from an interval $I' \subset I$ of diameter $\asymp \varepsilon |I|/k$.
This is the discrete analogue of a \textbf{porous set} in geometric measure theory.
It is well-known that porous sets do not have full Hausdorff dimension (e.g. see Chapter 11 in the book \cite{mattila1999geometry}), and our argument mimics this fact. 
The lower bound follows from a Cantor set construction.
We construct $\Dcal \subset [N]$ by removing the middle $4N/k$-sized interval $I$ of $[N]$ at each iteration.
We argue that for any $(1+\varepsilon)$-bi-Lipschitz map $f:[k] \to \Dcal$, the image of $f$ must lie entirely in only one of the two subintervals of $[N]-I$, and continue recursively.

\subsection{Upper bound (via a porosity argument)}

\begin{proof}
Throughout the proof, we will fix $N \in \n$ which is sufficiently large, and $\Dcal \subset [N]$ which has density $\delta \in (0,1)$ (i.e. $|\Dcal| \geq \delta N$). We will assume that there does not exist a $(1+\varepsilon)$ bi-Lipschitz embedding $f:[k] \to \Dcal$, and our goal is to prove an upper bound on the size $|\Dcal|$.

\textbf{Step 1: A single induction step:}
\\Consider the $k$ equidistant points $p_1,\ldots,p_k$ in the middle thirds subinterval of $N$.
$$p_i := i \dfrac{N}{3k} + \dfrac{N}{3}\;\;\;\text{for all }i=1,...,k.$$
For simplicity of notation we will ignore any floor functions in this section.
Consider now the interval centered at $p_1,...,p_k$ of small diameter:
$$I_i := [p_i - \dfrac{\varepsilon}{18 k} N, p_i + \dfrac{\varepsilon}{18 k} N]\;\;\;\text{for all }i=1,...,k.$$
It is routine to check that any map $f: [k] \to [N]$ with $f(i) \in I_i$ for all $i=1,...,k$ is a $(1+\varepsilon)$ bi-Lipschitz embedding of the path space $[k]$.
By the hypothesis on $\Dcal$ this means that for some $i \in [k]$, $I_i \cap \Dcal = \emptyset$.
We can therefore write $[N]-I_i = I^0 \cup I^1$ where $I^0, I^1$ are disjoint intervals of diameter $>N/3$.
(Note that we removed an interval of density $\asymp \varepsilon/k$ and not an interval of density $\asymp 1/k$. This is important in the estimate of Step 4.)
This completes the base case in the induction.

\textbf{Step 2: Iterating the inductive step:} 
\\We now apply the same argument to $I^0 \cap \Dcal \subset I^0$ and $I^1 \cap \Dcal \subset I^1$.
We pick $k$ equidistant points on the middle thirds subinterval of $I^0$ and place intervals of radius $\dfrac{\varepsilon}{18 k} |I^0|$ at each point.
By the hypothesis, one of these intervals is disjoint from $\Dcal$.
This means that $\Dcal \cap I^0$ lies in the union of two disjoint subintervals $I^{00}$ and $I^{01}$ each of diameter $>N/9$.

\textbf{Step 3: Estimating the total number of iterations:}
\\We apply this argument $M:=\log_3(N/3k)$ times.
Observe that after $M$ iterations, every interval has at least $k$ points and therefore we have enough points to pick $p_1,...,p_k$ and the intervals $I_1,...,I_k$ in the next iteration.

\textbf{Step 4: Estimating the size of $\Dcal$:}
\\At the $m^{\text{th}}$-iteration for each $m=1,...,M$, we conclude that the set $\Dcal$ lies inside the union of $2^m$ intervals. Denote this union by $U_m$.
Because we remove a large subinterval from each interval in the $m^{\text{th}}$-iteration, we have $|U_{m+1}| \leq (1-\dfrac{\varepsilon}{18k})|U_m|$.
We conclude:
 $$|\Dcal| \leq |U_M| \leq \left(1-\dfrac{\varepsilon}{18k}\right)^M N = \left(1-\dfrac{\varepsilon}{18k}\right)^{-\log_3(3k)} \left(1-\dfrac{\varepsilon}{18k}\right)^{\log_3 N} N = \left(1-\dfrac{\varepsilon}{18k}\right)^{-\log_3(3k)} N^{1+\log_3(1-\varepsilon/18k)}.$$
Since for all $m \in \n$, we have $(1-1/m)^m \geq 1/4$, we have
$$\left(1-\dfrac{\varepsilon}{18k}\right)^{-\log_3(3k)} \leq 4^{(\varepsilon/18k) \log_3(3k)} = \exp\left[\left(\dfrac{\log 4}{18 \log 3}\right) \dfrac{\varepsilon \log k}{k}\right] \leq e^{\varepsilon \log k /10k}.$$
Next, by the concavity of the logarithm at $1$:
$$\log_3(1-\varepsilon/18k) < -\dfrac{1}{18 \ln 3} \dfrac{\varepsilon}{k} < -\dfrac{\varepsilon}{50 k} $$
Rearranging the terms in the inequality, we get:
$$\delta \leq |\Dcal|/N \leq e^{\varepsilon \log k / 10k} N^{-\varepsilon/50k} \iff N^{-\varepsilon/50k} \leq \dfrac{1}{\delta}\;e^{\varepsilon \log k / 10k}$$
$$ \iff \log N \leq \dfrac{50k}{\varepsilon}\left(\log(1/\delta) + \dfrac{\varepsilon \log k}{10k}\right) = 50 \dfrac{k \log(1/\delta)}{\varepsilon} + 10 \log k < 100 \dfrac{k \log(1/\delta)}{\varepsilon}.$$
\end{proof}

\subsection{Lower bound (via a Cantor set construction)}\label{subSec:ProofLowerPathCantor}

\begin{proof}\;

\textbf{Step 1: Construction of the Cantor set}.
\\Throughout the proof we will fix $N \in \n$ which is sufficiently large, and construct a sequence of finer and finer approximations $[N]=U_0 \supset U_1 \supset ... \supset U_M =: \Dcal$. For each $m=1,...,M$, $U_m$ will be the disjoint union of $2^m$ intervals.
\\At the first step, we divide $[N] = I^0 \cup I^{mid} \cup I^1$ where
$$I^0 := [1, N(1/2-2/k)],\;\;\;I^{mid}:= [N(1/2-2/k), N(1/2+2/k)]\;\;\;I^1:=[N(1/2+2/k),N],$$
and set $U_1:=I^0 \cup I^1$. In plain words, we remove the middle interval of length $(4/k)N$.
\\At the second step, we remove the middle interval of length $(4/k)|I^0|$ from $I^0$ to obtain $I^{00}, I^{01}$ and do the same for $I^{1}$ and set $U_2:= I^{00} \cup I^{01} \cup I^{10} \cup I^{11}$. At the $m^{\text{th}}$ step, we take each of the $2^{m}$ intervals of $U_m$ and remove in each the middle interval of density $4/k$ to obtain $U_{m+1}$.
Observe that each interval in $U_m$ has diameter $<(2/3)^m N$. We iterate this construction $M = \log_{3/2}(N)$ steps, which ensures that the $2^M$ intervals which comprise $U_M$ are all singletons.

\textbf{Step 2: Analysis to forbid an embedding at a single scale}
\\\textbf{Claim}: if $f:[k] \to I^0 \cup I^1$ has distortion $\leq 2$, then either $f([k]) \subset I^0$ or else $f([k]) \subset I^1$.
\\Suppose not. Then there exists $i \in [k]$ such that $f(i) \in I^0$ while $f(i+1) \in I^1$.
This implies that $|f(i)-f(i+1)| > 4N/k = 4N/k |i-(i+1)|$. On the other hand, either $I^0$ or $I^1$, say $I^0$ has $|f([k]) \cap I^0| \geq k/2$.
By pigeonholing (cover $I^0$ with $k/2$ balls of radius $N/k$ each), there exists $j \neq j' \in [k]$ such that $|f(j)-f(j')| \leq 2N/k \leq 2N/k|j-j'|$.
We obtain the distortion bound:
$$Distortion(f) \geq \dfrac{|f(i)-f(i+1)|/|i-(i+1)|}{|f(j)-f(j')|/|j-j'|} > \dfrac{4N/k}{2N/k} = 2$$
and the claim follows.
Similarly, we can apply this claim to any interval in the Cantor set construction.  For every $m=1,...,M$ and any $\alpha \in \{0,1\}^m$ if we have an embedding $f: [k] \to I^{\alpha 0} \cup I^{\alpha 1}$ with distortion $\leq 2$, then either $f([k]) \subset I^{\alpha 0}$ or else $f([k]) \subset I^{\alpha 1}$.

\textbf{Step 3: Recursive argument to forbid any embedding:}
\\Suppose that $f:[k] \to \Dcal$ has distortion $\leq 2$.
Since $\Dcal \subset I^0 \cup I^1$, by the above claim, either $f([k]) \subset I^0$ or else $f([k]) \subset I^1$. Denote by $I^{\alpha_1}$ this interval.
By the same reasoning, since $\Dcal \cap I^{\alpha_1} \subset I^{\alpha_1 0} \cup I^{\alpha_1 1}$, either $f([k]) \subset I^{\alpha_1 0}$ or else $f([k]) \subset I^{\alpha_1 1}$. Denote this interval by $I^{\alpha_1\alpha_2}$.
By iteratively applying the claim $M$-times, we conclude that $f([k]) \subset I^{\alpha}$ for some $\alpha \in \{0,1\}^M$.
However, by construction $I^{\alpha}$ has diameter $< N (2/3)^M = 1$ meaning that $I^{\alpha}$ is a singleton.
A $2$-bi-Lipschitz map must be a bijection so we arrive at a contradiction.

\textbf{Step 4: Estimating the size of $\Dcal$:}
\\We now prove a lower bound on the size of the Cantor set $\Dcal$.
We have: $N = |U_0|$ and $|U_{m+1}| = (1-4/k)|U_m|$ so we get:
 $$\delta N := |\Dcal| = (1-4/k)^M N = N^{1 + \ln(1-4/k)/\ln(3/2)} > N^{1-(4/\ln(3/2))(1/k)}$$
where at the last step we used the concavity of the logarithm.

Rearranging the terms, we get: $N \geq e^{C' k \log(1/\delta)}$ where $C' = \ln(3/2)/4$.
\end{proof}

\section{\textbf{Metric embeddings of binary trees into dense subsets of binary trees}}\label{sec:treesRamsey}

In this section we prove Theorem \ref{thm:treeRamsey}. We restate the theorem below with explicit constants.

\begin{thmB}[Metric embeddings of trees into dense subsets of trees]
\;\\For all $k\geq 3, 0 < \delta < 1, 0 < \varepsilon \leq 0.001$, we have:
    $$e^{0.1 k \log(1/\delta)} \leq Tree(1+\varepsilon,\delta,k) \leq e^{1000 k \varepsilon^{-1} \log(1/\delta)}.$$
\end{thmB}

The upper bound uses as black boxes the upper bound for path spaces in Theorem \ref{thm:pathSpaces} and the Ramsey theorem for tree replicas of Pach--Solymosi--Tardos (namely Theorem 1 in \cite{pach2011remarks}), and the rest of the proof is analogous to the proof of the Furstenberg--Weiss theorem in \cite{pach2011remarks}.
The lower bound uses the same Cantor set construction as in the lower bound of Theorem \ref{thm:pathSpaces}, along with a characterization of path embeddings $[k] \to Tree(N)$ of small distortion. 

\subsection{Upper bound (via the tree replicas theorem of Pach--Solymosi--Tardos)}

\;\\We first recall some terminology about binary trees from \cite{pach2011remarks}. Let $N', N \in\n$ with $N' \leq N$.

A \textbf{tree replica} is a map $f: Tree(N') \to Tree(N)$ which satisfies two properties.
\\--First, if $w_1,w_2 \in Tree(N')$ which lie on the same level of the tree $Tree(N')$, i.e. $l(w_1) = l(w_2)$, then $f(w_1)$ and $f(w_2)$ lie in the same level of the tree $Tree(N)$, i.e. $l(f(w_1)) = l(f(w_2))$.
\\--Secondly, for all $w \in Tree(N')$ with immediate children $w_0$ and $w_1$, if we denote by $f(w)_0$ and $f(w)_1$ the immediate children of $f(w)$, then $f(w_0)$ is a descendant of $f(w)_0$ and $f(w_1)$ is a descendant of $f(w)_1$ (in vague words, the image of the tree "branches out" immediately after each point).

In \cite{pach2011remarks}, Pach--Solymosi--Tardos proved the following Ramsey theorem about tree replicas (Theorem 1).
They combined their result with Szemerédi's theorem to give quantitative bounds on the Furstenberg--Weiss theorem.
Denote by $h: [0,1/2] \to [0,1]$ the Shannon binary entropy function 
$$h(p) := p \log(1/p) + (1-p) \log(1/(1-p)).$$

\begin{thmB}[Pach--Solymosi--Tardos]\label{thm:Pach--Solymosi--Tardos}
    \;\\For all $N \in \n$ and $0<\delta<1$, any subset $\Dcal \subset Tree(N)$ with $\mu(\Dcal) \geq \delta$
    \\contains an arithmetic tree replica $f: Tree(N') \to \Dcal$ with $N' \geq \lfloor h^{-1}(\delta) N\rfloor$.
\end{thmB}

\begin{proof}
We are given $\Dcal \subset Tree(N)$ with $\mu(\Dcal) \geq \delta$.
By Theorem \ref{thm:Pach--Solymosi--Tardos} we obtain a tree replica: $f: Tree(N') \to \Dcal$ where $N' := h^{-1}(\delta) N$.

Similar to \cite{pach2011remarks}, we now consider the set of possible levels of $f(Tree(N')) \subset \Dcal$.
Call this set $\Dcal_{path} \subset [N]$, and observe that it has density $|\Dcal_{path}|/N = h^{-1}(\delta)$.
By the upper bound of Theorem \ref{thm:pathSpaces}, we can find an $\varepsilon$-distorted copy of the path $[k]$ sitting inside $\Dcal_{path}$ so long as:
$$N \leq e^{100 k \varepsilon^{-1} \log(1/h^{-1}(\delta))} \leq e^{1000 k \varepsilon^{-1} \log(1/\delta)}.$$
In the last inequality, we used the quantitative bound (and took the logarithm):
$$h^{-1}(\delta) \geq \delta^{2}\;\;\;\text{for all }0\leq \delta \leq 1/2$$

We consider only the levels of $Tree(N')$ which map one of the $k$ levels of the path $[k]$ in $\Dcal_{path}$.
This gives us a map $f':Tree(k) \to \Dcal$ which is an $\varepsilon$-distorted embedding.
\end{proof}

\subsection{Lower bound (via stability of path metrics in trees)} 
\;

We first describe the $\delta$-dense subset of the tree which avoids tree copies.
The description is much easier than the analysis.
Let $\Dcal_{path} \subset [N]$ be the Cantor set which was constructed in subsection \ref{subSec:ProofLowerPathCantor}.
We take $\Dcal \subset Tree(N)$ to be the set of all points $w \in Tree(N)$ whose level $l(w)$ lies in this set $\Dcal_{path}$.
Clearly, $\mu(\Dcal) = \mu(\Dcal_{path})$.
The lower bound follows once we show that $Tree(k)$ cannot embed into $\Dcal$ with small distortion.
In fact, we will show something stronger: the path metric space $[k]$ cannot embed into $\Dcal$ with small distortion.

Recall that a \textbf{geodesic} in a graph is a shortest path between two points (which is always unique in a tree). A geodesic between $w,w' \in Tree(N)$ has a very special form: the geodesic starts at $w$, travels upward the tree until the least common ancestor $w\vee w'$ of $w$ and $w'$ is visited, and then moves downwards again to $w'$.
We need the following stability lemma:\\

\begin{lemB}[Every small distortion path in a tree is close to a geodesic]
    \;\\\textbf{If} $f: [k] \to Tree(N)$ is a $0.001$-distorted embedding with rescaling $r \in [N]$, 
    \\and $\gamma \subset Tree(N)$ is the geodesic from $f(1)$ to $f(k)$ in $Tree(N)$,
    \\\textbf{then} every point in the image of $f$ is close to $\gamma$:
    $$d_T(f(i),\gamma) := \min_{w \in \gamma} d_T(f(i),w) \leq 0.001 r\;\;\;\text{for all }i\in[k].$$
\end{lemB}
\begin{proof}
    \textbf{Step 1: notation for tripods in trees}: 
    \\Given any triple of points $w,w',w'' \in Tree(N)$ we can form a \textbf{tripod} which is the union of all the geodesics $w \to w'$, $w' \to w''$ and $w \to w''$.
    Because of the tree structure, this tripod consists of a \textbf{median point} $m$ such that each of the $3$ geodesics must pass through $m$ (e.g. the geodesic $w \to w'$ is the geodesic $w \to m$ concatenated by the geodesic $m \to w'$).
    \\\textbf{Step 2: tripods of $3$ consecutive points must be degenerate}:
    \\For each $i \in [k-2]$ consider the points $f(i),f(i+1),f(i+2)$ and form their tripod. Denote the median by $g(i+1)$.
    We claim that because $f(i), f(i+1), f(i+2)$ form a slightly distorted copy of the path of length $3$, this forces $f(i+1)$ and $g(i+1)$ to be very close to each other. In particular, we have the estimates:
    $$d_T(f(i),f(i+2)) = d_T(f(i),g(i+1)) + d_T(g(i+1),f(i+2)) \geq 2r$$
    $$d_T(f(i),f(i+1)) = d_T(f(i),g(i+1)) + d_T(g(i+1),f(i+1)) \leq (1.001) r$$
    $$d_T(f(i+1),f(i+2)) = d_T(f(i+1),g(i+1)) + d_T(g(i+1),f(i+2)) \leq (1.001) r$$
    adding the last two inequalities and subtracting the first, we get:
    \begin{equation}\label{eq:proximity-to-tree-geodesic}
        2 d_T(f(i+1),g(i+1)) \leq 0.002 r \implies d_T(f(i+1),g(i+1)) \leq 0.001 r.
    \end{equation}
    \textbf{Step 3: gluing all the tripods together}:
    \\Set $g(1):=f(1)$ and $g(k):=f(k)$.
    \begin{claimB}
        For all $i \in [k-2]$, the geodesic from $g(i)$ to $g(i+2)$ must pass through the point $g(i+1)$
    \end{claimB}
    \begin{proof}
    This is a tedious routine check. Recall that by definition, both $g(i)$ and $g(i+1)$ belong to the unique geodesic from $f(i) \to f(i+1)$ and both $g(i+1)$ and $g(i+2)$ belong to the unique geodesic from $f(i+1) \to f(i+2)$.
    Consider the geodesics from $g(i) \to f(i+1)$ and from $g(i+2) \to f(i+1)$, and let $m$ be the first common point (i.e. $m$ is the median of the tripod formed by $\{g(i),g(i+2),f(i+1)\}$).
    The path joining the geodesics from $f(i) \to g(i) \to m \to g(i+2) \to f(i+2)$ is also a geodesic and since $g(i+1)$ is the median of the tripod formed by $\{f(i), f(i+1), f(i+2)\}$ it must be the case that $m=g(i+1)$.
    This shows the claim.
    \end{proof}
    Applying the above claim we conclude that the path $\gamma$ from $f(1) = g(1) \to g(2) \to ... \to g(k) = f(k)$ is a geodesic from $f(1)$ to $f(k)$. (If not, then there exists an edge in the tree which is crossed for two consecutive times. Because the path is a union of geodesics, this event can only happen when the edge is crossed at the end of a geodesic from $g(i) \to g(i+1)$ and then the same edge is crossed again at the beginning of the geodesic from $g(i+1) \to g(i+2)$. This contradicts the lemma, because in that event the geodesic from $g(i) \to g(i+2)$ does not pass through $g(i+1)$.)
    Together with the bound \ref{eq:proximity-to-tree-geodesic} we get the desired conclusion.
\end{proof}

We are now in a position to prove the lower bound in Theorem \ref{thm:treeRamsey}.

\begin{proof}
Suppose $f: [k] \to Tree(N)$ is an embedding of distortion $<0.001$.
Then by the above lemma, we take the geodesic from $f(1)$ to $f(k)$ and conclude that each point in the image, say $f(i)$ is $0.01r$-close to a point on the geodesic $g(i)$.
By the remark made earlier, half of the geodesic is level-decreasing and the other half is level increasing.
One of the two pieces must have at least $k/2$ of the points $g(1),...,g(k)$.
Without loss of generality say it is the first half of the path, so $g(1), g(2),...,g(k/2)$ is level decreasing.
It is a little tedious but not hard to check that the set of levels $\{l(f(1)),...,l(f(k/2))\} \subset[N]$ forms a $2$-distorted path of length $k/2$ in $[N]$.
By the lower bound of Theorem \ref{thm:pathSpaces} this forces $N$ to be large:
$$N \geq Path(2,\delta,k/2) = e^{\Omega(k \log(1/\delta))}.$$
\end{proof}

\section{\textbf{Large subsets of the Hamming cube have trivial BMW type}}\label{sec:BMWTypeApplication}

In order to prove Theorem \ref{thm:applicationToBMWType} we will use the following variation of Theorem \ref{thm:MAIN}.

\begin{thmB}\label{thm:variation_of_MAIN}
    \;\\There exist universal constants $c,c'>0$ such that for all $0<\gamma <1$ and $N \in \n$ with $N \geq c \gamma^{-3/2}$,
    \\\textbf{If} $\Dcal \subset Q_N$ has size
    $|\Dcal| > 2^{(1-\gamma)N}$,
    \\\textbf{then} $\Dcal$ contains a $2$-distorted copy of the Hamming cube of dimension $\geq c'/\sqrt{\gamma}$.
\end{thmB}
\begin{proof}
    We will denote by $k\in \n$ the dimension of the Hamming cube we wish to embed.
    We will need another error parameter $0<\varepsilon_3<1$ (not to be confused with the error terms $\varepsilon_1$ and $\varepsilon_2$ in Theorem \ref{thm:epsilon_distortion}).
    
    \textbf{Step 1: Inflating the set $\Dcal$ by $\varepsilon_3$ to reach density $1/2$:}
    \\We will inflate the set $\Dcal$ by an error $\varepsilon_3 \asymp 1/k$ to cover $1/2$ of the density and then we will apply Theorem \ref{thm:epsilon_distortion} to argue that $\Dcal$ is close to a roughly equitable block copy.
    \\Suppose that $|\Dcal|\geq |Ball(pN)|$ for some $0<p<1$. 
    \\By Harper's inequality (Lemma \ref{lemm:HarperIsoperimetry}), if $pN + \varepsilon_3 N> N/2$ then $\mu(\Dcal_{\varepsilon_3}) > 1/2$, which is what we desire.
    \\We need to estimate, for all $0<\gamma<1$, the largest value of $p$ such that 
    $$|\Dcal| \geq 2^{(1-\gamma)N} > |Ball(pN)|.$$
    We apply the Hoeffding's bounds (Lemma \ref{lem:Hoeffding_bounds_for_tails}) with $\varepsilon = \sqrt{2 \ln(2) \gamma}$ to get:
    \begin{equation}\label{eq:TailBoundForBMW}
    \left|Ball\left(\dfrac{1-\sqrt{2 \ln(2) \;\gamma}}{2} N\right)\right| \leq 2^N e^{-(2\ln(2)\gamma)(N/2)} = 2^{(1-\gamma)N}. 
    \end{equation}
    We take $\varepsilon_3 := \sqrt{\ln(2) \gamma/2}$. Harper's inequality and the bound (\ref{eq:TailBoundForBMW}) give us: $\mu(\Dcal_{\varepsilon_3}) \geq 1/2$.
    \\We choose the largest $k$ such that $\varepsilon_3 < 1/1000k$, i.e. $k = \lfloor \dfrac{\sqrt{2}}{1000 \sqrt{\ln 2}} \gamma^{-1/2}\rfloor + 1$.
    
    \textbf{Step 2: Finding a roughly equitable block copy of large rescaling:}
    \\By Theorem \ref{thm:epsilon_distortion} with $\varepsilon=1/2$ and $\delta = 1/2$ we get a roughly equitable block copy with block error $\varepsilon_1 = 1/100$ and rescaling $r = N/2k(1-\varepsilon_1 - 4 \varepsilon_2 k)$ inside the twice-inflated set $\Dcal_{\varepsilon_3 + \varepsilon_2}$, where $\varepsilon_2 = 1/1000k$.
    Written in the notation of section \ref{subSec:concentrationMeasureFacts}, we have:
    $\mathcal{V}_{\pm\varepsilon_1}^{k,n}\bigl(\mathcal{D}_{\varepsilon_2 + \varepsilon_3}\bigr)\neq\emptyset$.
    Applying Proposition \ref{prop:subspace_in_Dε_then_subspace_in_D_two_eps} 
    with perturbation error  $\varepsilon_2 + \varepsilon_3 \leq 1/(100k)$ and block error $\varepsilon_1 = 1/100$, conclude that $\Dcal$ contains a $10$-distorted copy of the Hamming cube $Q_k$ of dimension $k \gtrsim 1/\sqrt{\gamma}$ and the proof is complete.
\end{proof}

We state and prove the contrapositive of Theorem \ref{thm:applicationToBMWType} since it is more convenient for the logic of the proof.

\begin{thmB}[Rewritten - large subsets of the Hamming cube have trivial BMW type]\label{thm:OLDNOTATIONapplicationToBMWType}
    \;\\There are universal constants $c,c'>0$ such that the following holds.
    \\For any $\gamma > 0$ and $N \geq c \gamma^{-3/2}$ and any subset $\Dcal \subset \{0,1\}^{N}$ of size $|\Dcal| \geq 2^{(1-\gamma)N}$ we have:
    \begin{equation}
       c_M(\Dcal) \geq c' T_p(M)^{-1} \gamma^{-\frac{1}{2}(1-\frac{1}{p})} 
    \end{equation}
    for any metric space $M$ of BMW type $p>1$ constant $T_p(M)$.
\end{thmB}

\begin{proof}
   By an argument of Enflo \cite{enflo1970nonexistence,enflo1978infinite} 
    for any metric space $M$ and any $k \in \n$, 
    every map $f: Q_k \to M$ has bi-Lipschitz distortion $\geq k^{1-1/p}/T_p(M)$.
    Therefore the theorem follows once we embed $Q_k \to \Dcal$ with distortion $< 2$ for some $k \gtrsim 1/\sqrt{\gamma}$.
    This is precisely the content of Theorem \ref{thm:variation_of_MAIN}.
\end{proof}

\section{\textbf{Union problems for metric type and universality and related remarks}}\label{sec:remarksProofs}

The \textbf{Hales--Jewett theorem} \cite{hales1963regularity} states that for every finite alphabet $\Acal$ and number of colors $r \in \n$ there exists $N \in \n$ such that every $r$-coloring of the set $\Acal^N$ contains a monochromatic combinatorial line.
A \textbf{combinatorial line} is given via $\{l(a): a\in \Acal\}$ where $l$ is a \textbf{variable word} of length $N$, i.e. $l \in (\Acal \cup \{v\})^N$ where $v$ is an extra letter which appears at least once in $l$, and $l(i) \in \Acal^N$ is the word obtained by substituting $i$ in place of the extra letter $v$.

The \textbf{density Hales--Jewett theorem} of Furstenberg--Katznelson \cite{FurstenbergKatznelson1991} states that for every finite alphabet $\Acal$ and $\delta>0$ there exists $N \in \n$ such that every subset $\Dcal \subset \Acal^N$ of density $\mu(\Dcal) := |\Dcal|/|\Acal|^N > \delta$ contains a combinatorial line.
The \textbf{multi-dimensional density Hales--Jewett theorem} (which follows from the density Hales--Jewett theorem) extracts from dense subsets of $\Acal^N$ a \textbf{combinatorial subspace} which is given by substituting a variable word in multiple variable letters $v_1,...,v_k$ (where each variable letter appears at least once).
Observe that combinatorial subspaces do not give metric embeddings $f: \Acal^k \to \Acal^N$ because one variable letter in $v_1,...,v_k$ might have a large number of frequencies in $l$ while another variable letter might have a small number of frequencies in $l$. See the monograph \cite{dodos2016ramsey} and the survey \cite{dodos2018density} for more on the density Hales--Jewett theorem, its variations, and its history.

\begin{remB}\label{usingDensityHalesJewett}
    From the density Hales--Jewett theorem, 
    \\it follows easily that $\Emb(1,\delta,k) < \infty$ for all $k \in \n$ and $0<\delta < 1$.
\end{remB}
\begin{proof}
Fix $k \in \n$ and $0<\delta < 1$ and $N$ which will be chosen later.
Let $\Dcal \subset Q_N$ be a subset of density $\mu(\Dcal) \geq \delta$.
We break a point $x \in Q_N$ into $N/k$-many blocks of size $k$, meaning that we view $\Dcal \subset Q_N = \Acal^{N/k}$ where our alphabet is $\Acal := Q_k$.
Clearly, $\Dcal$ is a subset of density $\delta$ and $|\Acal| = 2^k$.
Let $N'$ be the dimension from the density Hales--Jewett theorem and choose $N:=kN'$.
By the density Hales--Jewett theorem there exists a 1-variable word $l \in (\Acal \cup \{v\})^N$ such that $l(\alpha) \in \Dcal$ for all $\alpha \in \Acal = Q_k$.
Observe $f: \alpha \mapsto l(\alpha)$ gives us a rescaled isometric embedding into $\Dcal$ and the rescaling $r$ is equal to the frequency of the variable letter $v$ in the variable word $l$.
\end{proof}

The same proof works for any $l_p$-grid $([m]^n,||\cdot||_p)$ with fixed thickness $m \in \n$ and growing dimension $n \in \n$.
It also gives following lemma which we need to prove Propositions \ref{prop::union_of_subspaces_and_metric_type} and \ref{prop::union_of_subspaces_and_metric_cotype}.

\begin{lemB}[Consequence of the Hales--Jewett theorem]\label{application_of_Hales--Jewett}
    \;\\For each $k,m \in \n$ there exists $n\in \n$ such that for all $N>n$ the following is true:
    \\(1) For any $2$-coloring (i.e. partition into $2$ pieces) of the Hamming cube $(Q_N,||\cdot||_1)$ there exists an isometric embedding with arbitrary rescaling 
    $$f: (Q_k,||\cdot||_1) \to (Q_N,||\cdot||_1)$$
    such that the image $f(Q_k)$ is monochromatic.
    \\(2) For any $2$-coloring (i.e. partition into $2$ pieces) of the $l_{\infty}$-grid $([m]^N,||\cdot||_\infty)$ there exists an isometric embedding with arbitrary rescaling 
    $$f: ([m]^k,||\cdot||_\infty) \to ([m]^N,||\cdot||_\infty)$$
    such that the image $f([m]^k)$ is monochromatic.
\end{lemB}
\begin{proof}
    We show only part $(2)$. Part $(1)$ follows from an identical argument (or alternatively from Remark \ref{usingDensityHalesJewett}).
    We break a point $x \in [m]^N$ into $N/k$-many blocks of size $k$, meaning that we view $[m]^N = \Acal^{N/k}$ where our alphabet is $\Acal := [m]^k$.
    Clearly, the $2$-coloring of $[m]^N$ gives us a $2$-coloring of $\Acal^{N/k}$ and our alphabet size is $|\Acal| = m^k$.
    Let $N'$ be the dimension from the Hales--Jewett theorem and choose $N:=kN'$.
    By the Hales--Jewett theorem,
    there exists a monochromatic 1-variable word $l \in (\Acal \cup \{v\})^N$ such that $l(\alpha)$ has the same color for all $\alpha \in \Acal = Q_k$.
    Observe $f: [m]^k \to [m]^N: \alpha \mapsto l(\alpha)$ gives us a monochromatic rescaled isometric embedding and the rescaling $r$ is equal to the frequency of the variable letter $v$ in the variable word $l$.
\end{proof}

Next, we prove Proposition \ref{prop::union_of_subspaces_and_metric_type}, Proposition \ref{prop:decomposition_into_positive_and_negative}, and Proposition
\ref{prop::union_of_subspaces_and_metric_cotype}.
We repeat the statements for the convenience of the reader.

\begin{propB}[The union of metric subspaces of nontrivial BMW type has nontrivial BMW type]
    \textbf{If} $(M,d)$ is a metric space with $M = A \cup B$ for some $A,B \subset M$
    \textbf{then}
    $$p_{A}>1\;\;\text{and}\;\;p_{B}>1 \implies p_M > 1.$$
\end{propB}

\begin{proof}
    Suppose that $p_M = 1$. We will show that either $p_A=1$ or else $p_B=1$.
    
    By a theorem of Bourgain--Milman--Wolfson, namely Theorem 2.6 in \cite{bourgain1986type}, $p_M = 1$ if and only if for all $n \in \n$ and $\varepsilon>0$, the Hamming cube of dimension $n$ embeds into $M$ with distortion $<1+\varepsilon$.

    For each $k \in \n$, let $n$ be the value from Part (1) of Lemma \ref{application_of_Hales--Jewett}.
    Consider the $2$-distorted copy $f: Q_n \to M$, and label each point $\alpha \in Q_n$ by $A$ if $f(\alpha) \in A$ and by $B$ if $f(\alpha) \in B$.
    By Lemma \ref{application_of_Hales--Jewett}, we have a monochromatic undistorted copy of $Q_k \to Q_n$.
    Label $k$ by $A$ if this copy is monochromatic with $A$ and label $k$ by $B$ if this copy is monochromatic with $B$.

    Either the set $\{k \in \n: k \;\text{has label }A\}$ is infinite or else the set $\{k \in \n: k \;\text{has label }B\}$ is infinite.
    Without loss of generality, suppose that $\{k \in \n: k \;\text{has label }A\}$ is infinite.
    This implies that for every $k \in \n$, $A$ contains a $2$-distorted copy of the Hamming cube of dimension $k$.
    
    By an argument of Enflo \cite{enflo1970nonexistence,enflo1978infinite} 
    for any metric space $M'$ and any $k \in \n$, 
    every map $f: Q_k \to M'$ has bi-Lipschitz distortion $\geq k^{1-1/p}/T_p(M')$.
    Combining these two observations, we conclude that $T_p(A) = \infty$ for all $p>1$, and hence $p_A = 1$.
\end{proof}

\begin{propB}[Failure to decompose into non-positive and non-negative curvature]
    \;\\For any $D>1$, there exists a metric space $(M,d)$ (namely the Hamming cube $Q_N$ for some $N \in \n$) which cannot split as the union of subsets $M = A \cup B$ such that 
    \\--$A$ embeds into a metric space of non-negative Alexandrov curvature with distortion $\leq D$ and
    \\--$B$ embeds into a metric space of non-positive Alexandrov curvature with distortion $\leq D$.
\end{propB}
\begin{proof}
    By a theorem of Ohta \cite{ohta2009markov}, a metric space of non-negative Alexandrov curvature has BMW type $2$ with constant $1+\sqrt{2}$.
    This implies that any map of the Hamming cube of dimension $k$ into a metric space of non-negative Alexandrov curvature has distortion $\geq \sqrt{k}/(1+\sqrt{2})$.
    
    On the other hand, any metric space of non-positive Alexandrov curvature has BMW type $2$ constant $=1$ (the proof essentially goes back to Enflo \cite{enflo1970nonexistence}, see e.g. \cite{ohta2009markov}).
    This implies that any map of the Hamming cube of dimension $k$ into a metric space of non-positive Alexandrov curvature has distortion $\geq \sqrt{k} > \sqrt{k}/(1+\sqrt{2})$.

    Find $k \in \n$ such that $\sqrt{k}/(1+\sqrt{2}) > D$.
    Apply part (A) of Lemma \ref{application_of_Hales--Jewett} to obtain the value $n \in \n$.

    Suppose that the Hamming cube of dimension $n$ can be decomposed as $Q_n = A \cup B$ where $A$ embeds into a metric space of non-negative Alexandrov curvature with distortion $\leq D$ and $B$ embeds into a metric space of non-positive Alexandrov curvature with distortion $\leq D$.
    Without loss of generality, the two sets $A$ and $B$ are disjoint.
    Label each point in $Q_n$ by $A$ or by $B$ accordingly.
    By Lemma \ref{application_of_Hales--Jewett}, there exists an isometric embedding $f: Q_k \to Q_n$ whose image is entirely contained inside $A$ or entirely contained inside $B$.
    By our choice of parameters $\sqrt{k}/(1+\sqrt{2}) > D$, we arrive at a contradiction.
\end{proof}

\begin{propB}[A universal metric space cannot be split into two non-universal subspaces]
    \;\\\textbf{If} $(M,d)$ is a metric space with $M = A \cup B$ for some $A,B \subset M$
    and $M$ is universal,
    \\\textbf{then} either $A$ is universal or else $B$ is universal.
\end{propB}
\begin{proof}
    Fix $\varepsilon>0$ throughout the entire proof.
    
    Since $M$ is universal, for all $n,m \in \n$ the $l_\infty$-grid $([m]^n,||\cdot||_\infty)$ embeds into $M$ with distortion $<1+\varepsilon$.
    Call this embedding $f_{m,n}$.
    For each $m,n \in \n$ we have a $2$-coloring of $[m]^n$ based on whether a point $\alpha \in [m]^n$ has $f(\alpha) \in A$ or $f(\alpha) \in B$.

    Fix $m, k \in \n$ and apply Lemma \ref{application_of_Hales--Jewett} for a sufficiently high $n \in \n$ and the $2$-coloring of $[m]^n$ obtained from $f_{m,n}$.
    We conclude that for each $m, k \in \n$, the $l_\infty$-grid $([m]^k, ||\cdot||_\infty)$ embeds either to $A$ with distortion $\leq 1+\varepsilon$ or to $B$ with distortion $\leq 1+\varepsilon$.
    We label the pair $(m,k)\in \n \times \n$ by $A$ or by $B$ accordingly.

    Label each $m \in \n$ by $A$ if $\{(m,k) \in \n\times \n: (m,k)\;\text{has label }A \}$ is infinite and label $m \in \n$ by $B$ otherwise.
    Either $\{m \in \n: m \;\text{has label }A\}$ is infinite or else $\{m \in \n: m \;\text{has label }B\}$ is infinite.
    Say it is the former.
    Because for each $m'\leq m$ and $k' \leq k$ we have an undistorted embedding
    $([m']^{k'}, ||\cdot||_\infty) \to ([m]^k, ||\cdot||_\infty)$,
    we conclude for every $m, k \in \n$, $([m]^k, ||\cdot||_\infty)$ embeds into $A$ with distortion $\leq 1+\varepsilon$.

    By the Fréchet embedding theorem, for every finite metric space $M_0$ there exists $k,m$ such that $M_0$ embeds into $([m]^k, ||\cdot||_\infty)$ with distortion $\leq 1+\varepsilon$.
    Therefore $M_0$ embeds into $A$ with distortion $\leq (1+\varepsilon)^2$. 
    This shows that $A$ is universal.
\end{proof}

\section{\textbf{Questions and directions for further study}}\label{sec:conclusion}

\begin{conjecture}
    There exists $C>0$ such that $\Emb(1,\delta,k) \leq e^{C k} \log 1/\delta$ for all $k \in \n$ and $0<\delta<1$.
\end{conjecture}

\begin{conjecture}
    For each fixed $0<\delta<1$ and $\varepsilon > 0$,
    \\the power type behavior of $\Emb(1+\varepsilon,\delta,k)$ as $k \to \infty$ is linear, that is:
    $$\limsup_{k \to \infty} \dfrac{\log(\Emb(1+\varepsilon,\delta,k))}{\log k} = \liminf_{k \to \infty} \dfrac{\log(\Emb(1+\varepsilon,\delta,k))}{\log k} = 1$$
\end{conjecture}

\begin{question}
    Prove a coloring version of Theorem \ref{thm:MAIN} and of Theorem \ref{thm:treeRamsey}. How do the bounds change?
\end{question}

\begin{question}[Dependence on the error]
    What is the dependence on the bi-Lipschitz error $\varepsilon$ in the path spaces and tree Ramsey problems?
    Are the upper bounds in Theorems \ref{thm:pathSpaces} and \ref{thm:treeRamsey} sharp up to universal constants?
    In other words do we have
    $$Path(1+\varepsilon, \delta, k) = e^{\Omega( k \varepsilon^{-1} \log(1/\delta))}\;\;\;\text{and}\;\;\;Tree(1+\varepsilon, \delta, k) = e^{\Omega( k \varepsilon^{-1} \log(1/\delta))}$$
    for all $k\in \n$, $0<\delta<1$ and $0 < \varepsilon < 0.01$? What is the dependence when the distortion $C$ is very large?
    Is it true that for all fixed $C>1$, $Path(C,\delta,k)$ and $Tree(C,\delta,k)$ grow exponentially in $k$?
\end{question}

\begin{question}[Cubes of fixed higher width]
    Fix a \textbf{width parameter} $m\in \n$ and an \textbf{exponent} $1\leq p \leq \infty$.
    For each $k \in \n$ and $0<\delta<1$ and $\varepsilon>0$, let $N(1+\varepsilon,k,\delta,m,p)$ be the smallest $N$ such that for every subset $\Dcal \subset [m]^N$ of size $|\Dcal|\geq \delta m^N$ there exists a $(1+\varepsilon)$ bi-Lipschitz embedding $f: ([m]^k,||\cdot||_p) \to (\Dcal,||\cdot||_p)$.
    An application of density Hales--Jewett similar to Remark \ref{usingDensityHalesJewett} shows that $N(1,k,\delta,m,p) < \infty$.
    \\Does the growth of $N(1+\varepsilon,k,\delta,m,p)$ as $k \to \infty$ depend on the choice of width $m$ or exponent $p$?
\end{question}

\begin{question}[Grids of fixed higher dimension]
    Fix a \textbf{grid dimension} $d \in \n$.
    For each $k \in \n$ and $0<\delta<1$ and $\varepsilon>0$, let $GRID_d(1+\varepsilon,k,\delta)$ be the smallest $N$ such that for all $\Dcal \subset [N]^d$ with $|\Dcal| \geq \delta N^d$ there exists a $(1+\varepsilon)$ bi-Lipschitz map $f: ([k]^d,||\cdot||_1) \to (\Dcal,||\cdot||_1)$.
    A modification of the proof of Theorem \ref{thm:pathSpaces} (with a $4$-cornered Cantor set for the lower bound) seems to give:
    $$e^{\Omega_d(k \log(1/\delta))} \lesssim GRID_d(1+\varepsilon,k,\delta) \lesssim e^{O_d(k^d \varepsilon^{-d} \log(1/\delta))}.$$
    For fixed $0 < \varepsilon \lesssim1$ and $0< \delta < 1$, how does $GRID_d(1+\varepsilon,k,\delta)$ grow as $k \to \infty$?
\end{question}

\begin{question}[Vague Direction]\label{question:vague_direction}
    Find more examples of sequences of finite metric measure spaces $(X_N,d_N,\mu_N)_{N=1}^\infty$ with the following property:
    for every $\delta>0$ and $k \in \n$, there exists $N \in \n$ such that for every $\Dcal \subset X_N$ there exists a \textit{metric embedding} $f: X_k \to \Dcal$.
    Here \textit{metric embedding} can mean either a $(1+\varepsilon)$ bi-Lipschitz map, an undistorted map, or an undistorted map of bounded rescaling.
    Investigate the dependence of $N$ on the parameters $k$, $\delta$ and $\varepsilon$.
\end{question}

\printbibliography

@book{ostrovskii2013metric,
  title={Metric embeddings: Bilipschitz and coarse embeddings into Banach spaces},
  author={Ostrovskii, Mikhail I},
  volume={49},
  year={2013},
  publisher={Walter de Gruyter}
}

@article{hales1963regularity,
  title={Regularity and positional games},
  author={Hales, Alfred W and Jewett, Robert I},
  journal={Transactions of the American Mathematical Society},
  volume={106},
  number={2},
  pages={222--229},
  year={1963}
}

@article{OstrovskiiRandrianantoanina+2022+313+329,
url = {https://doi.org/10.1515/agms-2022-0145},
title = {On L1-Embeddability of Unions of L1-Embeddable Metric Spaces and of Twisted Unions of Hypercubes},
author = {Mikhail I. Ostrovskii and Beata Randrianantoanina},
pages = {313--329},
volume = {10},
number = {1},
journal = {Analysis and Geometry in Metric Spaces},
doi = {10.1515/agms-2022-0145},
year = {2022},
lastchecked = {2026-06-04}
}

@article{guentner2013discrete,
  title={Discrete groups with finite decomposition complexity},
  author={Guentner, Erik W and Tessera, Romain and Yu, Guoliang},
  journal={Groups, Geometry, and Dynamics},
  volume={7},
  number={2},
  pages={377--402},
  year={2013}
}

@article{bell2001asymptotic,
  title={On asymptotic dimension of groups},
  author={Bell, Gregory and Dranishnikov, Alexander N},
  journal={Algebraic \& Geometric Topology},
  volume={1},
  number={1},
  pages={57--71},
  year={2001},
  publisher={Mathematical Sciences Publishers}
}

@article{MakarychevMakarychev,
author = {Makarychev, Konstantin and Makarychev, Yury},
year = {2016},
month = {02},
pages = {},
title = {A Union of Euclidean Metric Spaces is Euclidean},
volume = {14},
journal = {Discrete Analysis},
doi = {10.19086/da.876}
}

@article{mendel2013ultrametric,
  title={Ultrametric skeletons},
  author={Mendel, Manor and Naor, Assaf},
  journal={Proceedings of the National Academy of Sciences},
  volume={110},
  number={48},
  pages={19256--19262},
  year={2013},
  publisher={National Academy of Sciences},
  doi={10.1073/pnas.1202500109},
  eprint={1112.3416},
  archivePrefix={arXiv}
}

@article{DadarlatGuentner+2007+1+15,
url = {https://doi.org/10.1515/CRELLE.2007.081},
title = {Uniform embeddability of relatively hyperbolic groups},
author = {Marius Dadarlat and Erik Guentner},
pages = {1--15},
volume = {2007},
number = {612},
journal = {Journal für die reine und angewandte Mathematik},
doi = {10.1515/CRELLE.2007.081},
year = {2007},
lastchecked = {2026-06-04}
}

@article {BFM86,
    AUTHOR = {Bourgain, J. and Figiel, T. and Milman, V.},
     TITLE = {On {H}ilbertian subsets of finite metric spaces},
   JOURNAL = {Israel J. Math.},
  FJOURNAL = {Israel Journal of Mathematics},
    VOLUME = {55},
      YEAR = {1986},
    NUMBER = {2},
     PAGES = {147--152},
      ISSN = {0021-2172},
   MRCLASS = {46C05 (46B20 54E35)},
  MRNUMBER = {868175},
MRREVIEWER = {M.\ I.\ Kadets},
       DOI = {10.1007/BF02801990},
       URL = {https://doi.org/10.1007/BF02801990},
}

@Misc{naor2022personal,
  author = {Naor, Assaf},
  month  = sep,
  title  = {Personal communication},
  year   = {2022},
}

@article{FurstenbergKatznelson1991,
  author  = {Furstenberg, Hillel and Katznelson, Yitzhak},
  title   = {A density version of the {H}ales--{J}ewett theorem},
  journal = {Journal d'Analyse Math\'ematique},
  year    = {1991},
  volume  = {57},
  number  = {1},
  pages   = {64--119},
  doi     = {10.1007/BF03041066}
}

@article{DK2022,
  author  = {Dodos, Pandelis and Karamanlis, Miltiadis},
  title   = {Forbidden sparse intersections},
  journal = {Forum of Mathematics, Sigma},
  volume  = {13},
  year    = {2025},
  pages   = {e99},
  doi     = {10.1017/fms.2025.10067}
}

@article{mendel2013ultrametricHausdorff,
  title     = {Ultrametric subsets with large {H}ausdorff dimension},
  author    = {Mendel, Manor and Naor, Assaf},
  journal   = {Inventiones Mathematicae},
  volume    = {192},
  number    = {1},
  pages     = {1--54},
  year      = {2013},
  publisher = {Springer},
  doi       = {10.1007/s00222-012-0402-7},
  eprint    = {1106.0879},
  archivePrefix = {arXiv},
}

@article{bartal2006ramsey,
  title     = {Ramsey-type theorems for metric spaces with applications to online problems},
  author    = {Bartal, Yair and Bollob{\'a}s, B{\'e}la and Mendel, Manor},
  journal   = {Journal of Computer and System Sciences},
  volume    = {72},
  number    = {5},
  pages     = {890--921},
  year      = {2006},
  publisher = {Elsevier},
  doi       = {10.1016/j.jcss.2005.05.008},
  eprint    = {cs/0406028},
  archivePrefix = {arXiv},
}

@book{bridson2013metric,
  title     = {Metric Spaces of Non-Positive Curvature},
  author    = {Bridson, Martin R. and Haefliger, Andr{\'e}},
  series    = {Grundlehren der mathematischen Wissenschaften},
  volume    = {319},
  year      = {1999},
  publisher = {Springer},
  address   = {Berlin},
  isbn      = {978-3-540-64324-1},
  doi       = {10.1007/978-3-540-64324-1},
}

@article{gromov2003random,
  title     = {Random walk in random groups},
  author    = {Gromov, Mikhail},
  journal   = {Geometric and Functional Analysis},
  volume    = {13},
  number    = {1},
  pages     = {73--146},
  year      = {2003},
  publisher = {Birkh{\"a}user-Verlag Basel},
  doi       = {10.1007/s000390300002},
}

@Article{gromov2001cat,
  author    = {Gromov, Mikhail Leonidovich},
  journal   = {Zapiski Nauchnykh Seminarov POMI},
  title     = {{CAT}($\kappa$)-spaces: construction and concentration},
  year      = {2001},
  number    = {0},
  pages     = {101--140},
  volume    = {280},
  publisher = {St.~Petersburg Department of Steklov Mathematical Institute, RAS},
  url       = {https://www.mathnet.ru/eng/znsl1463},
}

@article{HARPER1966385,
  title   = {Optimal numberings and isoperimetric problems on graphs},
  journal = {Journal of Combinatorial Theory},
  volume  = {1},
  number  = {3},
  pages   = {385--393},
  year    = {1966},
  issn    = {0021-9800},
  doi     = {10.1016/S0021-9800(66)80059-5},
  url     = {https://www.sciencedirect.com/science/article/pii/S0021980066800595},
  author  = {L.H. Harper},
}

@article{enflo1970nonexistence,
  title     = {On the nonexistence of uniform homeomorphisms between {$L^p$}-spaces},
  author    = {Enflo, Per},
  journal   = {Arkiv f{\"o}r Matematik},
  volume    = {8},
  number    = {2},
  pages     = {103--105},
  year      = {1970},
  publisher = {Springer},
  doi       = {10.1007/BF02589549},
}

@article{ohta2009markov,
  title         = {Markov type of {A}lexandrov spaces of non-negative curvature},
  author        = {Ohta, Shin-Ichi},
  journal       = {Mathematika},
  volume        = {55},
  number        = {1-2},
  pages         = {177--189},
  year          = {2009},
  publisher     = {London Mathematical Society},
  doi           = {10.1112/S0025579300001005},
  eprint        = {0707.0102},
  archivePrefix = {arXiv},
}

@article{kondo2012cat,
  title         = {{CAT}(0) spaces and expanders},
  author        = {Kondo, Takefumi},
  journal       = {Mathematische Zeitschrift},
  volume        = {271},
  number        = {1},
  pages         = {343--355},
  year          = {2012},
  publisher     = {Springer},
  doi           = {10.1007/s00209-011-0866-y},
}

@article{ball1992markov,
  title     = {Markov chains, {R}iesz transforms and {L}ipschitz maps},
  author    = {Ball, Keith},
  journal   = {Geometric \& Functional Analysis},
  volume    = {2},
  number    = {2},
  pages     = {137--172},
  year      = {1992},
  publisher = {Springer},
  doi       = {10.1007/BF02099193},
}

@article{rodl2022blurred,
  title         = {A blurred view of {V}an der {W}aerden type theorems},
  author        = {R{\"o}dl, Vojtech and Sales, Marcelo},
  journal       = {Combinatorics, Probability and Computing},
  volume        = {31},
  number        = {4},
  pages         = {684--701},
  year          = {2022},
  publisher     = {Cambridge University Press},
  doi           = {10.1017/S0963548321000237},
}

@article{fraser2021approximate,
  title         = {Approximate arithmetic structure in large sets of integers},
  author        = {Fraser, Jonathan M. and Yu, Han},
  journal       = {Real Analysis Exchange},
  volume        = {46},
  number        = {1},
  pages         = {163--174},
  year          = {2021},
  doi           = {10.14321/realanalexch.46.1.0163},
  eprint        = {1905.05034},
  archivePrefix = {arXiv},
}

@article{fraser2018arithmetic,
  title         = {Arithmetic patches, weak tangents, and dimension},
  author        = {Fraser, Jonathan M. and Yu, Han},
  journal       = {Bulletin of the London Mathematical Society},
  volume        = {50},
  number        = {1},
  pages         = {85--95},
  year          = {2018},
  publisher     = {Wiley Online Library},
  doi           = {10.1112/blms.12112},
  eprint        = {1611.06960},
  archivePrefix = {arXiv},
}

@Article{dumitrescu2010approximate,
author       = {Adrian Dumitrescu},
  title        = {Approximate Euclidean Ramsey Theorems},
  journal      = {J. Comput. Geom.},
  volume       = {2},
  number       = {1},
  pages        = {16--29},
  year         = {2011},
  url          = {https://doi.org/10.20382/jocg.v2i1a2},
  doi          = {10.20382/JOCG.V2I1A2},
  bibsource    = {dblp computer science bibliography, https://dblp.org}
}

@book{mattila1999geometry,
  title     = {Geometry of Sets and Measures in {E}uclidean Spaces: Fractals and Rectifiability},
  author    = {Mattila, Pertti},
  series    = {Cambridge Studies in Advanced Mathematics},
  number    = {44},
  year      = {1995},
  publisher = {Cambridge University Press},
  address   = {Cambridge},
  isbn      = {978-0-521-65595-4},
  doi       = {10.1017/CBO9780511623813},
}

@book{dodos2016ramsey,
  title     = {Ramsey Theory for Product Spaces},
  author    = {Dodos, Pandelis and Kanellopoulos, Vassilis},
  series    = {Mathematical Surveys and Monographs},
  volume    = {212},
  year      = {2016},
  publisher = {American Mathematical Society},
  address   = {Providence, RI},
  isbn      = {978-1-4704-2808-8},
  doi       = {10.1090/surv/212},
}

@inproceedings{dodos2018density,
  title     = {Density {H}ales--{J}ewett numbers---where do we stand?},
  author    = {Dodos, Pandelis},
  booktitle = {Proceedings of the 16th Panhellenic Conference on Mathematical Analysis},
  pages     = {11},
  year      = {2018},
}

@book{alon2016probabilistic,
  title     = {The Probabilistic Method},
  author    = {Alon, Noga and Spencer, Joel H.},
  year      = {2016},
  edition   = {4},
  publisher = {John Wiley \& Sons},
  address   = {Hoboken, NJ},
  isbn      = {978-1-119-06195-3},
  doi       = {10.1002/9781119422532},
}

@article{enflo1978infinite,
  title   = {On infinite-dimensional topological groups},
  author  = {Enflo, Per},
  journal = {S{\'e}minaire Maurey-Schwartz (1975--1976), Espaces $L^p$, applications radonifiantes et g{\'e}om{\'e}trie des espaces de Banach},
  pages   = {1--11},
  year    = {1978},
  note    = {Exp.~No.~10-11, {\'E}cole Polytechnique, Paris},
}

@article{mendel2007ramsey,
  title         = {Ramsey partitions and proximity data structures},
  author        = {Mendel, Manor and Naor, Assaf},
  journal       = {Journal of the European Mathematical Society},
  volume        = {9},
  number        = {2},
  pages         = {253--275},
  year          = {2007},
  doi           = {10.4171/JEMS/79},
  eprint        = {cs/0511084},
  archivePrefix = {arXiv},
}

@article{shelah1972combinatorial,
  title     = {A combinatorial problem; stability and order for models and theories in infinitary languages},
  author    = {Shelah, Saharon},
  journal   = {Pacific Journal of Mathematics},
  volume    = {41},
  number    = {1},
  pages     = {247--261},
  year      = {1972},
  publisher = {Mathematical Sciences Publishers},
  doi       = {10.2140/pjm.1972.41.247},
}

@article{sauer1972density,
  title     = {On the density of families of sets},
  author    = {Sauer, Norbert},
  journal   = {Journal of Combinatorial Theory, Series A},
  volume    = {13},
  number    = {1},
  pages     = {145--147},
  year      = {1972},
  publisher = {Elsevier},
  doi       = {10.1016/0097-3165(72)90019-2},
}

@inproceedings{lovett2011bounded,
  title         = {Bounded-depth circuits cannot sample good codes},
  author        = {Lovett, Shachar and Viola, Emanuele},
  booktitle     = {Proceedings of the 26th Annual IEEE Conference on Computational Complexity (CCC'11)},
  pages         = {243--251},
  year          = {2011},
  organization  = {IEEE},
  doi           = {10.1109/CCC.2011.11},
}

@article{benjamini2016bi,
  title         = {Bi-{L}ipschitz bijection between the {B}oolean cube and the {H}amming ball},
  author        = {Benjamini, Itai and Cohen, Gil and Shinkar, Igor},
  journal       = {Israel Journal of Mathematics},
  volume        = {212},
  number        = {2},
  pages         = {677--703},
  year          = {2016},
  publisher     = {Springer},
  doi           = {10.1007/s11856-016-1302-0},
  eprint        = {1310.2017},
  archivePrefix = {arXiv},
}

@article{mendel2008metric,
  title         = {Metric cotype},
  author        = {Mendel, Manor and Naor, Assaf},
  journal       = {Annals of Mathematics},
  volume        = {168},
  number        = {1},
  pages         = {247--298},
  year          = {2008},
  doi           = {10.4007/annals.2008.168.247},
  eprint        = {math/0506201},
  archivePrefix = {arXiv},
}

@article{bourgain1986type,
  title   = {On type of metric spaces},
  author  = {Bourgain, Jean and Milman, Vitali and Wolfson, Haim},
  journal = {Transactions of the American Mathematical Society},
  volume  = {294},
  number  = {1},
  pages   = {295--317},
  year    = {1986},
  doi     = {10.1090/S0002-9947-1986-0819949-8},
}

@article{naor2012introduction,
  title         = {An introduction to the {R}ibe program},
  author        = {Naor, Assaf},
  journal       = {Japanese Journal of Mathematics},
  volume        = {7},
  number        = {2},
  pages         = {167--233},
  year          = {2012},
  publisher     = {Springer},
  doi           = {10.1007/s11537-012-1222-7},
  eprint        = {1204.2193},
  archivePrefix = {arXiv},
}

@article{pach2011remarks,
  title         = {Remarks on a {R}amsey theory for trees},
  author        = {Pach, J{\'a}nos and Solymosi, J{\'o}zsef and Tardos, G{\'a}bor},
  journal       = {Combinatorics, Probability and Computing},
  volume        = {22},
  number        = {5},
  pages         = {750--763},
  year          = {2013},
  publisher     = {Cambridge University Press},
  doi           = {https://doi.org/10.1007/s00493-012-2763-3},
  eprint        = {1107.5301},
  archivePrefix = {arXiv},
}

@article{furstenberg2003markov,
  title     = {Markov processes and {R}amsey theory for trees},
  author    = {Furstenberg, Hillel and Weiss, Benjamin},
  journal   = {Combinatorics, Probability and Computing},
  volume    = {12},
  number    = {5-6},
  pages     = {547--563},
  year      = {2003},
  publisher = {Cambridge University Press},
  doi       = {10.1017/S0963548303005893},
}

@article{szemeredi1975sets,
  title     = {On sets of integers containing $k$ elements in arithmetic progression},
  author    = {Szemer{\'e}di, Endre},
  journal   = {Acta Arithmetica},
  volume    = {27},
  number    = {1},
  pages     = {199--245},
  year      = {1975},
  publisher = {Polska Akademia Nauk. Instytut Matematyczny PAN},
  doi       = {10.4064/aa-27-1-199-245},
}

@article{BartalLinialMendelNaor2005,
  author        = {Bartal, Yair and Linial, Nathan and Mendel, Manor and Naor, Assaf},
  title         = {On metric {R}amsey-type phenomena},
  journal       = {Annals of Mathematics},
  year          = {2005},
  volume        = {162},
  number        = {2},
  pages         = {643--709},
  doi           = {10.4007/annals.2005.162.643},
  eprint        = {math/0406353},
  archivePrefix = {arXiv},
}

@book{MitzenmacherUpfal2005,
  author    = {Mitzenmacher, Michael and Upfal, Eli},
  title     = {Probability and Computing: Randomized Algorithms and Probabilistic Analysis},
  publisher = {Cambridge University Press},
  address   = {Cambridge},
  year      = {2005},
  doi       = {10.1017/CBO9780511813603},
}
\end{document}